\documentclass[12pt,reqno]{amsart}
\usepackage{soul}
\usepackage{multicol,amsmath,graphics,cancel,soul,color,bbm,amsfonts,dsfont,tikz,mathrsfs}
\usepackage[french,english]{babel}
\usepackage[left=1in, right=1in, top=1.1in,bottom=1.1in]{geometry}
\usetikzlibrary{matrix,patterns,positioning}
\numberwithin{equation}{section}

\newcommand{\Cc}{\mathcal{C}}

\newcommand{\Fc}{\mathcal{F}}
\newcommand{\Gc}{\mathcal{G}}
\newcommand{\Hc}{\mathcal{H}}

\newcommand{\Pc}{\mathcal{P}}

\newcommand{\D}{\mathbb{D}}
\newcommand{\E}{\mathbb{E}}

\newcommand{\N}{\mathbb{N}}
\newcommand{\Pb}{\mathbb{P}}

\newcommand{\R}{\mathbb{R}}

\newcommand{\Hg}{\mathfrak{H}}

\newcommand{\Sf}{\mathscr{S}}
\newcommand{\Ef}{\mathscr{E}}
\newcommand{\ud}{\mathrm{d}}

\newcommand{\Norm}[1]{\left\lVert#1\right\rVert}
\newcommand{\Abs}[1]{\left|#1\right|}
\newcommand{\Ip}[1]{\left\langle #1 \right\rangle}

\newcommand{\Indi}[1]{\mathbbm{1}_{#1}}

\newtheorem{thm}{Theorem}[section]

\newtheorem{definition}[thm]{Definition}

\newcounter{dummy} \numberwithin{dummy}{section}
\newtheorem{Definition}[dummy]{Definition}
\newtheorem{Proposition}[dummy]{Proposition}
\newtheorem{Theorem}[dummy]{Theorem}

\newtheorem{Lemma}[dummy]{Lemma}

\newtheorem{Remark}[dummy]{Remark}

\def\1{{\rm l}\hskip -0.21truecm 1}

\begin{document}
\title[Fluctuations for matrix-valued Gaussian processes]{Fluctuations for matrix-valued Gaussian processes}
\date{\today}
\address{Mario Diaz: Instituto de Investigaciones en Matem\'{a}ticas Aplicadas y en Sistemas, Universidad Nacional Aut\'{o}noma de M\'{e}xico, Mexico City, Mexico.}
\email{mario.diaz@sigma.iimas.unam.mx}
\address{Arturo Jaramillo: Centro de Investigaci\'on en Matem\'aticas A.C., Calle Jalisco s/n, CP 36240, Guanajuato, Mexico.}
\email{jagil@cimat.mx}
\address{Juan Carlos Pardo: Centro de Investigaci\'on en Matem\'aticas A.C., Calle Jalisco s/n, CP 36240, Guanajuato, Mexico.}
\email{jcpardo@cimat.mx}


\date{\today}

\author{Mario Diaz, Arturo Jaramillo, Juan Carlos Pardo}
\date{\today}


\selectlanguage{english}
\begin{abstract}
We consider a symmetric matrix-valued  Gaussian  process $Y^{(n)}=(Y^{(n)}(t);t\ge0)$ and its empirical spectral measure process $\mu^{(n)}=(\mu_{t}^{(n)};t\ge0)$. Under some mild conditions on the covariance function of $Y^{(n)}$, we find an explicit expression for the limit distribution of
$$Z_F^{(n)} := \left( \big(Z_{f_1}^{(n)}(t),\ldots,Z_{f_r}^{(n)}(t)\big) ; t\ge0\right),$$
where $F=(f_1,\dots, f_r)$, for $r\ge 1$, with each component belonging to a large class of test functions, and
$$ Z_{f}^{(n)}(t) := n\int_{\R}f(x)\mu_{t}^{(n)}(\ud x)-n\E\left[\int_{\R}f(x)\mu_{t}^{(n)}(\ud x)\right].$$
More precisely, we establish the stable convergence of $Z_F^{(n)}$  and determine its limiting distribution.  An upper bound for the total variation distance  of the law of $Z_{f}^{(n)}(t)$ to its limiting distribution, for a test function $f$ and $t\geq0$ fixed, is also given.
\end{abstract}
\maketitle
\selectlanguage{french}
\begin{abstract}
Nous consid\'erons un processus gaussien sym\'etrique \`a valeurs matricielles  $Y^{(n)}=(Y^{(n)}(t);t\ge0)$ et son processus des mesures spectrales empiriques $\mu^{(n)}=(\mu_{t}^{(n)};t\ge0)$. Dans des conditions assez faibles sur la fonction de covariance de  $Y^{(n)}$ nous trouvons une expression explicite pour la distribution limite de
$$Z_F^{(n)} := \left( \big(Z_{f_1}^{(n)}(t),\ldots,Z_{f_r}^{(n)}(t)\big) ; t\ge0\right),$$
o\`u $F=(f_1,\dots, f_r)$,  pour $r\ge 1$ o\`u chaque composant appartient \`a une grande classe des fonctions test, et 
$$ Z_{f}^{(n)}(t) := n\int_{\R}f(x)\mu_{t}^{(n)}(\ud x)-n\E\left[\int_{\R}f(x)\mu_{t}^{(n)}(\ud x)\right].$$
Plus pr\'ecisément, nous \'etablissons la convergence stable de $Z_F^{(n)}$  et nous d\'eterminons sa distribution limite. Une borne sup\'erieure pour la distance en variation totale de la loi de $Z_{f}^{(n)}(t)$  vers sa distribution limite, pour une fonction test $f$  et $t\geq0$ fixe , est \'egalement donn\'ee.
\end{abstract}

\section{Introduction}\label{sec:intro}
For a given positive integer $n$, we denote by $\R^{n\times n}$ the set of real matrices of dimension $n\times n$ and consider a sequence of processes $Y^{(n)}=(Y^{(n)}(t);\ t\geq0)$, taking values in $\R^{n\times n}$, defined in a given probability space $(\Omega, \Fc,\Pb)$. Assume that for every $n\in\N$ and $t\geq0$, the random matrix $Y^{(n)}(t)=[Y_{i,j}^{(n)}(t)]_{1\leq i,j\leq n}$ is real and symmetric whose entries are determined by
\begin{align}\label{eq:Y}
Y_{i,j}^{(n)}(t)
  &=\left\{\begin{array}{cc} \frac{1}{\sqrt{n}}X_{i,j}(t)& \text{ if }\ i< j,\\\frac{\sqrt{2}}{\sqrt{n}}X_{i,i}(t)& \text{ if }\ i=j,\end{array}\right.
\end{align}
where $X_{i,j}=(X_{i,j}(t);\ t\geq0)$, for $i\leq j$, is a collection of i.i.d.~centered Gaussian processes with covariance function $R(s,t)$. Namely, the processes $\{X_{i,j};i\leq j\}$ are jointly Gaussian, centered and satisfy
\begin{equation*}
\phantom{\text{for }i\leq j\text{ and }l\leq k,} \quad \quad \E\left[X_{i,j}(t)X_{l,k}(s)\right] =\delta_{i,l}\delta_{j,k}R(s,t), \qquad \text{for }i\leq j\text{ and }l\leq k,
\end{equation*}
where $\delta_{i,l}$ denotes the Kronecker delta, i.e., $\delta_{i,l} = 1$ if $i=l$ and $\delta_{i,l} = 0$ otherwise.  For convenience, we assume without loss of generality that $R(1,1)=1$. Due to well known distributional symmetries exhibited by $Y^{(n)}(t)$, see, e.g., \cite[Sec.~2.5]{anderson2010introduction}, in the sequel we  refer to $Y^{(n)}$ as a Gaussian Orthogonal Ensemble process, or GOE process for short. We denote by $\lambda_{1}^{(n)}(t)\geq \cdots\geq \lambda^{(n)}_{n}(t)$ the ordered eigenvalues of $Y^{(n)}(t)$ and by $ \mu_t^{(n)}$ its associated empirical spectral distribution, defined by
 \[
 \mu_t^{(n)}(\ud x)=\frac{1}{n}\sum_{i=1}^{n}\delta_{\lambda_{i}^{(n)}(t)}(\ud x),
 \]
where $\delta_{z}(\ud x)$ denotes the Dirac measure centered at $z$, i.e., the probability measure characterized by $\delta_z(\{z\}) = 1$. 

\noindent This manuscript extends to a second-order level the recent work by Jaramillo et al.~\cite{jaramillo2018convergence} where the convergence in probability, under  the topology of uniform convergence over compact sets, of the empirical spectral measure processes $\mu^{(n)} := (\mu_t^{(n)};t\geq0)$  is established and  its limit is characterized in terms of its Cauchy transform. Our goal, here, is to  provide a functional central limit theorem for the process 
\begin{align}\label{eq:fluct2}
\left(n\int_{\R}F(x)\mu_t^{(n)}(\ud x)-n\E\bigg[\int_{\R}F(x)\mu_t^{(n)}(\ud x)\bigg]\ ;\ t\geq0\right),
\end{align}
where $F:\R\rightarrow\R^r$ is a sufficiently regular test function and $r\geq 1$. In order to setup an appropriate context for stating our main results (see Section~\ref{Sec:contrib}), we review briefly some of the  literature related to the study of the asymptotic properties of the process \eqref{eq:fluct2}.

\noindent Our starting point is the celebrated Wigner Theorem \cite{wigner1955characteristic,wigner1957characteristic}, which asserts that for every $\varepsilon>0$ and every element $f$ belonging to the set $C_b(\R)$ of continuous and bounded functions,
\begin{equation}
\label{eq:WignerThm}
\lim_{n\to\infty} \Pb\left(\left|\int_\R f(x) \mu_1^{(n)}(\ud x) - \int_\R f(x) \mu_1^{sc}(\ud x)\right| > \epsilon \right) = 0,
\end{equation}
where $\mu_\sigma^{sc}$, for $\sigma>0$, denotes the scaled semicircle distribution 
\begin{align*}
\mu_{\sigma}^{sc}(\ud x)
  &:=\frac{\Indi{[-2\sigma,2\sigma]}(x)}{2\pi\sigma^2}\sqrt{4\sigma^2-x^2}\ud x.
\end{align*}
In other words, we have that $\mu_1^{(n)}$ converges weakly in probability to the standard semicircle distribution. 
Since its publication, Wigner's theorem has been generalized and extended in many different directions.
Given that the aim of this paper is the study the asymptotic law of \eqref{eq:fluct2}, we now recall some developments regarding the fluctuations of $\int_\R f(x) \mu_1^{(n)}(\ud x)$ around its mean and those that describe the properties of the sequence of measure-valued processes $(\mu_{t}^{(n)} ; t\geq0)$.

 Despite the fact that our paper deals exclusively with  GOE processes, we also mention for the sake of completeness, some representative results on other type of ensembles, with special emphasis on the Gaussian Unitary Ensemble process, GUE process for short. That is to say,  a matrix-valued process whose construction is analogous to that of $Y^{(n)}$, with the exception that $Y^{(n)}(t)$ is Hermitian for all $t\geq0$, the factor $\sqrt{2}$ appearing in \eqref{eq:Y} is replaced by 1 and the real Gaussian processes $X_{i,j}(t)$, for $i<j$,  are replaced by complex Gaussian processes whose real and imaginary parts are independent copies of $X_{1,1}$.

In the GOE case, the problem of studying $\mu_{t}^{(n)}$, as a function of the variable $t\geq0$, was first addressed  by Rogers and Shi \cite{rogers1993interacting}, and  C\'epa and L\'epingle  \cite{cepa1997diffusing} in the specific case when  the processes $X_{i,j}$'s are standard Brownian motions. More recently, when the $X_{i,j}$'s are Gaussian processes,  Jaramillo et al.~\cite{jaramillo2018convergence} proved that under some mild conditions on the covariance function $R(s,t)$, the sequence of measure-valued processes $(\mu_t^{(n)}; t\geq 0)$ converges in probability to the process $(\mu_{R(t,t)^{\frac{1}{2}}}^{sc}; t\geq0)$, in the topology of uniform convergence over compact sets.

The above results can be seen as a type of law of large numbers, thus it is   natural to ask about the fluctuations of random variables of the form $\int_{\R}f(x)\mu_1^{(n)}(\ud x),$ with $f:\R\rightarrow\R$ belonging to a set of suitable test functions. This problem was originally addressed by Johansson in \cite{johansson1998fluctuations} for a random matrix model which includes the GOE and GUE as special cases. Afterwards, Lytova and Pastur \cite{lytova2009central} studied the fluctuations of general Wigner matrices satisfying a Lindeberg type condition. In particular, the authors in \cite{lytova2009central} proved that 
\begin{align}\label{eq:fluctuationss}
n\int_{\R}f(x)\mu_1^{(n)}(\ud x)-n\E\bigg[\int_{\R}f(x)\mu_1^{(n)}(\ud x)\bigg]
  &\stackrel{d}{\rightarrow}\mathcal{N}(0,\sigma_{f}^2),\ \ \ \ \text{as } n\rightarrow\infty,
\end{align}
where $\stackrel{d}{\rightarrow}$ denotes convergence in law and $\mathcal{N}(0,\sigma_{f}^2)$ is a centered Gaussian random variable with variance given by
\begin{align*}
\sigma_{f}^2
  &:=\frac{1}{4}\int_{\R^{2}}\left(\frac{f(x)-f(y)}{x-y}\right)^2\frac{4-xy}{(4-x^2)(4-y^2)}\mu_1^{\text{sc}}(\ud x)\mu_1^{\text{sc}}(\ud y).
\end{align*}
In addition to \cite{lytova2009central}, there have been many results related to the study of the limit in distribution  \eqref{eq:fluctuationss}, for instance Anderson and Zeitouni \cite{anderson2006clt}, Bai and Yao \cite{bai2005convergence}, Cabanal-Duvillard \cite{cabanal2001fluctuations},  Chatterjee \cite{chatterjee2009fluctuations}, Girko \cite{girko2001theory}, Guionnet \cite{guionnet2002large}, to name but a few. The techniques that have been used for this purpose are quite diverse, for instance  Johansson  \cite{johansson1998fluctuations} addresses the problem by using the joint density of $\lambda_1^{(n)},\dots, \lambda_n^{(n)}$.  Bai and Yao  \cite{bai2005convergence} used  the Cauchy-Stieltjes transform to reduce the problem  to the case where 
$$f(x)=\frac{1}{x-z},$$
for $z$ belonging to the upper complex plane. The approach followed by Lytova and Pastur \cite{lytova2009central}  consists on using the Fourier transform and an interpolation method, while the one introduced  by Cabanal-Duvillard  \cite{cabanal2001fluctuations} relies on stochastic calculus techniques.

%

On the other hand, the fluctuations of the process \eqref{eq:fluct2} with $f:\R\rightarrow\R$ belonging to a set of suitable test functions
have not been deeply studied. Indeed,  the study of \eqref{eq:fluct2} in the GOE regime has been restricted to the case where the entries of $Y^{(n)}$ are Ornstein-Uhlenbeck processes. For this case, it was proved by Israelson \cite{israelsson2001asymptotic} that not only \eqref{eq:fluct2} converges weakly to a Gaussian process, but also the process of signed measures $(n(\mu_{t}^{(n)}-\mu_t^{sc}) ;\ t\geq0)$ converges in law to a distribution-valued Gaussian process. Although \cite{israelsson2001asymptotic} established existence and uniqueness of the  limit law, it was not  characterized explicitly. 
This problem was addressed by Bender \cite{bender2008global} where the asymptotic covariance function for the limit of $(n(\mu_{t}^{(n)}-\mu_t^{sc}) ;\ t\geq0)$  was derived and its law was implicitly characterized. The case where the entries of $Y^{(n)}$ are Brownian motions has not  been addressed  yet in the GOE regime, although there are some partial results for the GUE case, as  discussed  below.

For the GUE process, the problem of determining the limit of $\{\mu^{(n)}\ ;\ n\geq 1\}$ has been only explicitly addressed for the case where $Y^{(n)}$ is a Dyson Brownian motion. That is to say, when the $X_{i,j}$'s, for $i<j$, are standard complex Brownian motions or equivalently, when the covariance function of $X_{1,1}$ is of the form $R(s,t)=s\wedge t$. For this type of matrices, the techniques from \cite{rogers1993interacting} and \cite{cepa1997diffusing} can  still be applied, leading to an analogous result as in the GOE case (the reader is referred to \cite[Section~4.3]{anderson2010introduction} for a complete proof of this fact).

 The problem of studying the limiting distribution of \eqref{eq:fluct2} in the GUE regime has been addressed, simultaneously with the GOE case, in the aforementioned papers \cite{anderson2006clt,bai2005convergence,bai2004clt,chatterjee2009fluctuations,guionnet2002large,johansson1998fluctuations}. Unfortunately, it   has been restricted to the cases where $Y^{(n)}$ is either a Dyson Brownian motion or an Ornstein Uhlenbeck matrix-valued process. For the Brownian motion case, it was proved by P\'{e}rez-Abreu and Tudor in \cite{perez2007functional} that the sequence of processes \eqref{eq:fluct2}
converges towards a Gaussian process in the topology of uniform convergence over compact sets. However, the shape of the covariance function of the limiting process was not described in an explicit closed form. The Ornstein Uhlenbeck matrix-valued  case was addressed simultaneously with the GOE regime in the aforementioned paper \cite{bender2008global}.

Finally,  we would like to mention some additional developments related to the fluctuations of other random matrix ensembles. For instance, we mention the work of  Guionnet \cite{guionnet2002large}, where among other things, a central limit theorem for Gaussian band matrix models is obtained. 
This result was later extended by Anderson and Zeitouni \cite{anderson2006clt} to the more general case of band matrix models whose on-or-above diagonal entries are independent but neither necessarily identically distributed nor necessarily all with  the same variance. The approach used in \cite{anderson2006clt} was based on combinatorial enumeration, generating functions and concentration inequalities. Another related topic is the one introduced by Diaconis  and Shahshahani \cite{diaconis1994eigenvalues}, and further developed by Diaconis and Evans \cite{diaconis2001linear}, which consists on the study of fluctuations of orthogonal, unitary and symplectic Haar matrices. The main tool that was used for solving this problem is the method of moments, but the computations are more complicated in comparison to the GOE and GUE case due to the lack of independence between the matrix entries. The study of the fluctuations of unitary matrices was further developed by L\'{e}vy and Ma\"{i}da \cite{levy2010central}. Specifically, the fluctuations of a Brownian motion on the unitary group were studied using tools from stochastic calculus, similar to those used by Cabanal-Duvillard in \cite{cabanal2001fluctuations}, and tools from second-order free probability introduced by Mingo, \'{S}niady and Speicher in \cite{mingo2007second}. More recently, C\'{e}bron and Kemp \cite{cebron2014fluctuations} took a geometric approach to study the fluctuations of a family of diffusion processes on the general linear group which includes both the standard and unitary Brownian motions as special cases. We would also like to mention the paper of Bai and Silverstein \cite{bai2004clt} which is devoted to the study of  the fluctuations of sample covariance matrices, as well as  the paper of Diaz et al.~\cite{diaz2017global}, where a central limit theorem for block Gaussian matrices is derived by means of a combinatorial analysis of the second-order Cauchy transform.  Last but not least,  we mention the work of Unterberger in \cite{unterberger2018global}, which is closely related to \cite{israelsson2001asymptotic} and \cite{bender2008global}, and deals with the problem of determining the asymptotic law of a suitable renormalization of the empirical distribution process of a generalized Dyson Brownian motion. As in \cite{israelsson2001asymptotic} and \cite{bender2008global}, it is proved in \cite{unterberger2018global} that the aforementioned renormalization converges to a distribution-valued Gaussian process, although an explicit expression for the covariance of the limiting Gaussian distribution was not provided.

\section{Main results}\label{Sec:contrib}
\noindent As we mentioned before and motivated by the aforementioned results, we devote this manuscript to prove a central limit theorem for \eqref{eq:fluct2} which holds for GOE processes with a general covariance function $R(s,t)$, where the fluctuations are parametrized by a time variable $t$ and a general vector valued test function $F$. As a consequence,  we also provide an upper bound  for the total variation  distance  of $Z_{f}^{(n)}(t)$ and its limit distribution. As an additional improvement, all the limit theorems presented here are stated in the context of stable convergence, which is an extension of the convergence in law first introduced by Renyi in \cite{Renyi1} and whose definition is given below.

\begin{Definition}\label{def:stable}
Assume that $\{\eta_n; n\ge 1\}$ is a sequence of random variables defined on $(\Omega, {\mathcal{F}}, \mathbb{P})$ with values on a complete and separable metric space $S$ and $\eta$ is an $S$-valued  random variable defined on the enlarged probability space $(\Omega, {\mathcal{G}}, \mathbb{P})$.  We say that $\eta_n$ {\em converges stably} to $\eta$ as $n \to \infty$, if for any continuous and bounded function $g:S \rightarrow {\mathbb R}$ and any ${\mathbb R}$-valued, ${\mathcal{F}}$-measurable bounded random variable $M$, we have
\begin{align}\label{eq:stableconveq}
\lim_{n\to \infty} {\mathbb E} \left[g(\eta_n) M\right]
  &= {\mathbb E}\left[g(\eta) M\right].
\end{align}
We denote the stable convergence of $\{\eta_n, n\ge 1\}$ towards $\eta$ by $\eta_n \stackrel{\mathcal{S}}{\longrightarrow}\eta$.
\end{Definition}
\noindent 
\begin{Remark}\label{ref:rem1}
Observe that by taking the variable $M$ identically equal to one in the definition above, one can easily deduce that the topology of stable convergence is finer than that of the weak convergence. In addition, by taking $\eta$ to be independent of $\mathcal{F}$ (this will be consistent with our use of definition \ref{def:stable}), the expectation in right hand side of \eqref{eq:stableconveq} splits the product in the form 
$$ {\mathbb E}\left[g(\eta) M\right]
= {\mathbb E}\left[g(\eta)]\E[ M\right]$$
giving a notion of asymptotic independence of the limiting random variable and the initial underlying $\sigma$-algebra $\mathcal{F}$. A powerful implementation of this instance of stable convergence arises when studying mixed Gaussian processes, as ilustrated in \cite{CorNuWoe}, \cite{BinNourNualart} and \cite{HaJaNualart}. 
\end{Remark}

In the absence of the martingale property, stochastic calculus approaches for the study of the fluctuations in \eqref{eq:fluct2} (e.g., \cite{bender2008global}, \cite{israelsson2001asymptotic}, \cite{perez2007functional}, \cite{unterberger2018global}) might not be straightforward to generalize. Specifically, Knight's theorem, which is the most common tool for deriving functional limit theorems, cannot be directly applied to processes arising from an underlying source of randomness that lacks of a martingale structure. A particular instance of such a process is the case where $X_{1,1}$ is a fractional Brownian motion with Hurst parameter different from $1/2$. To overcome this difficulty, we use techniques  from the theory of Malliavin calculus which  have been quite effective for studying limit distributions of functionals of Gaussian processes, see for instance, the monograph of Nourdin and Peccati \cite{nourdin2012normal} for a presentation of the recent advances in these topics. We would also like to emphasize that the results here presented  are only proved for real symmetric matrices while those considered in \cite{cabanal2001fluctuations} and \cite{perez2007functional} hold for complex Hermitian matrices. It is worth mentioning that although a Malliavin calculus methodology can be established in the GUE regime as well, such techniques must be applied to each entry of the underlying stochastic matrix, which causes the computations to increase too much in their complexity when passing from orthogonal to unitary ensembles. This makes the GUE calculations intractable to us. For this reason, we have decided to focus momentarily on the GOE case, leaving the study of the remaining ensembles as an interesting open problem for future research.

 In order to present our main results, we introduce the following notation. 
For a fixed covariance function $R(s,t)$, we define its associated standard deviation $\sigma_s$, and correlation coefficient $\rho_{s,t}$, by 
\begin{align}\label{eq:Rfunctionals}
\sigma_{s}
  &:=\sqrt{R(s,s)}
	\qquad\textrm{and}\qquad
\rho_{s,t}:=\frac{R(s,t)}{\sigma_{s}\sigma_{t}}, \qquad \textrm{for}\quad t,s\ge 0.
\end{align}
Consider the set of test functions
\begin{align}\label{eq:Pmathcaldef}
\mathcal{P}
  &:=\{f\in\Cc^4(\R;\R)\ :\, f^{(4)}\ \text{ has polynomial growth}\}.
\end{align}
For $f\in \mathcal{P}$, let $Z^{(n)}_f=(Z_f^{(n)}(t); t\geq0)$ be given by
\begin{align}\label{eq:Ztfn}
Z_f^{(n)}(t)
  &:=n\left(\int_{\R}f(x)\mu_{t}^{(n)}({\rm d}x)-\E\left[\int_{\R}f(x)\mu_{t}^{(n)}({\rm d}x)\right]\right).
\end{align}
Similarly, if $\Pc^{r}$ denotes the $r$-th cartesian product of $\Pc$ and $F:=(f_{1},\dots, f_r)\in\Pc^r$, we define the process $Z_{F}^{(n)}=(Z_F^{(n)}(t);  t\geq0)$, as follows
\begin{align*}
Z_F^{(n)}(t)
  &:=\left(Z_{f_{1}}^{(n)}(t),\dots, Z_{f_{r}}^{(n)}(t)\right).
\end{align*}
Our goal consists on determining asymptotic properties of the law of $Z_F^{(n)}(t)$. We would like to briefly comment on the generality of the family of test functions $\mathcal{P}$ that we are utilizing. First we observe that $\Pc$ contains all functions $f$ of the form  
\begin{enumerate}
    \item $f(x)=\mathfrak{R}(\frac{1}{z-x})$ and $f(x)=\mathfrak{I}(\frac{1}{z-x})$, for $z$ taken  in the upper complex semi plane and where $\mathfrak{R}(z)$ and $\mathfrak{I}(z)$ denote the real and imaginary parts of $z$, respectively. In particular, $Z_F^{(n)}(t)$ can be taken as a tuple of Cauchy transforms for the spectral empirical distributions $\mu_{t}^{(n)}$.
    \item $f(x)=p(x)$, where a polynomial in several variables with arbitrary degree. In particular, $Z_F^{(n)}(t)$ can be taken as a tuple of linear combination of mixed moments of $\mu_{t}^{(n)}$).
\end{enumerate}
The choice of the set $\Pc$ becomes natural when we face the problem of proving the property of sequential compactness for $Z_F^{(n)}(t)$, presented in Section \ref{sec:tight}. More precisely, the fact that we make use of the stochastic equation \cite[Lemma~3.1]{jaramillo2018convergence} (which involves derivatives of order one), as well as a mean value theorem for the difference quotient appearing within such equation (which  induces the requirement of an additional derivative on the underlying test function) and a second order Malliavin derivative that is applied to one of the terms of the resulting object (see Equation \eqref{eq:D2Gftinc}), lead us to consider the set $\Pc$ as the most convenient candidate for the family of test functions. We clarify however, that we don't have any evidence indicating that the conditions required for the test functions are sharp, and it is possible that the results we present below can be extended to a framework that allows more generality on $\Pc$. A key step in determining the limit law of $Z_F^{(n)}(t)$, as  $n$ increases,  consists on describing the asymptotic behaviour of the covariance 
\begin{align}\label{eq:limitcovunk}
\lim_{n\rightarrow\infty}\text{Cov}\Big[Z_{f}^{(n)}(s),Z_{g}^{(n)}(t)\Big],
\end{align}
for $f,g\in\Pc$ and $s,t>0$. This problem was addressed by Pastur and Shcherbina in \cite{pastur2011eigenvalue} for the case  $s=t$, where it was proved that 
\begin{align}\label{eq:PasturSherbina}
\lim_{n\rightarrow\infty}\text{Cov}\big[Z_f^{(n)}(s),Z_g^{(n)}(s)\Big]
  &=\frac{1}{2\pi}\int_{[-2\sigma_s,2\sigma_s]^2}\frac{\Delta f}{\Delta\lambda}\frac{\Delta g}{\Delta\lambda}\frac{4\sigma_{s}^2-\lambda_1\lambda_2}{\sqrt{4\sigma_{s}^2-\lambda_1^2}\sqrt{4\sigma_{s}^2-\lambda_2^2}}\ud\lambda_1\ud\lambda_2,
\end{align}
where $\Delta f:=f(\lambda_1)-f(\lambda_2)$ and $\Delta\lambda:=\lambda_1-\lambda_2$. 
Up to our knowledge, there is  no analog of the formula \eqref{eq:PasturSherbina} for  $s\neq t$, so we have devoted Section \ref{sec:covariance} to the development of a new technique for studying  the limit \eqref{eq:limitcovunk}. Our approach is also  based on Malliavin calculus together with properties of Chebyshev polynomials and free Wigner integrals. The aforementioned objects have been crucial for the study of limit theorems obtained as functionals of a free Brownian motion (se for instance \cite{kemp2012wigner}). However, the use of free calculus for studying the asymptotic covariances of linear statistics is presented for the first time in this manuscript and has an interest on its own. In order to make this more precise, lets introduce some notation. Let $U_{q}$ denote the $q$-th Chebyshev polynomial of second order in $[-2,2]$, characterized by the property 
\begin{align}\label{eq:Chebyshevdef}
U_{q}(2\cos(\theta))
  &=\frac{\sin((q+1)\theta)}{\sin(\theta)}.
\end{align}
In Lemma \ref{lem:K}, we prove that for all $-2<x,y<2$ and $0\leq z<1$, the series
\begin{align}\label{eq:Kdef0}
K_{z}(x,y)
  &:=\sum_{q=0}^{\infty}U_{q}(x)U_{q}(y)z^{q},
\end{align}
is absolutely convergent, non-negative, and satifies
\begin{align}\label{eq:kerKdef}
K_{z}\big(x,y\big)
  &=\frac{1-z^2}{z^2(x-y)^2-xyz(1-z)^2+(1-z^2)^2}.
\end{align} 
The limit \eqref{eq:limitcovunk} can then be expressed in terms of $K_{z}$, as it is  indicated below.

\begin{Theorem}\label{theorem:covariance}
Let $\rho_{s,t}$ and $\sigma_{s}$ be given as in \eqref{eq:Rfunctionals}. Then, for $f,g\in \mathcal{P}$, 
\begin{align}\label{eq:covasymptotic}
\lim_{n\rightarrow\infty}\mathrm{Cov}\Big[Z_{f}^{(n)}(s),Z_{g}^{(n)}(t)\Big]
  &=2\int_{\R^{2}}f^{\prime}(x)g^{\prime}(y)\nu^{\rho_{s,t}}_{\sigma_s,\sigma_t}(\ud x,\ud y),
\end{align}
where the measure $\nu^{\rho_{s,t}}_{\sigma_s,\sigma_t}$ is absolutely continuous with respect to the Lebesgue measure, with density $f^{\rho_{s,t}}_{\sigma_s,\sigma_t}(x,y)$, given by
\[
f^{\rho_{s,t}}_{\sigma_s,\sigma_t}(x,y)
  :=\left\{\begin{array}{ll}
	\frac{\sqrt{4\sigma_s^2-x^2}\sqrt{4\sigma_t^2-y^2}}{2\pi^2\sigma_s^2\sigma_t^2}\displaystyle\int_{0}^{1}
	K_{z\rho_{s,t}}(x/\sigma_s,y/\sigma_t)\ud z,&
\mathrm{if }\,(x,y)\in I_{s,t},\\ 0 & \mathrm{otherwise,}
\end{array}\right .
\]
where $I_{s,t}=[-2\sigma_s,2\sigma_s]\times [-2\sigma_t,2\sigma_t]$.
\end{Theorem}

The proof of Theorem~\ref{theorem:covariance} is deferred to Section~\ref{sec:covariance}. It relies on tools from Malliavin calculus and Voiculescu's free probability theory, both subjects are reviewed in Section~\ref{sec:chaos}.

\noindent In the sequel, we  assume that $R$ satisfies the following regularity conditions:
\begin{itemize}
\item[\textbf{(H1)}] There exists $\alpha>1$, such that for all $T>0$ and $t\in[0,T]$, the mapping $s\mapsto R(s,t)$ is absolutely continuous on $[0,T]$, and 
\begin{align*}
\sup_{0\leq t\leq T}\int_{0}^{T}\Abs{\frac{\partial R}{\partial s}(s,t)}^{\alpha}\ud s<\infty.
\end{align*}
\item[\textbf{(H2)} ]The map $s\mapsto \sigma_s^2=R(s,s)$ is continuously differentiable in $(0,\infty)$ and continuous at zero. Moreover, there exists $\varepsilon\in(0,1)$, such that the mapping $s\mapsto s^{1-\varepsilon}R^\prime(s,s)$ is bounded over compact intervals of $\R$.
\end{itemize}
\noindent As a direct consequence of \textbf{(H2)}, we have that $|R^\prime(s,s)|$ is integrable in a neighborhood of zero.  We observe that the conditions above are very mild, so the collection of processes satisfying \textbf{(H1)} and \textbf{(H2)} includes processes with very rough trajectories, such as fractional Brownian motion with Hurst parameter $H\in(0,1)$,  whose covariance function is of  the form
\begin{equation}\label{BMfrac}
    R(s,t)
  =\frac{1}{2}(s^{2H}+t^{2H}-|t-s|^{2H}),
\end{equation}
and  trajectories are H\"older continuous of order $\alpha\in(0,H)$.  In order to shed some light on the nature of the conditions we have imposed over the covariance of $X$, we anticipate to the reader that part of the proofs of our main results (in particular, the sequential compactness of the normalized linear statistics \eqref{eq:fluctuationss}) rely on techniques of stochastic integration against rough Gaussian processes, which suggests the use of \textbf{(H1)} for guaranteeing the well posedness of the so-called generalized Skorohod integral for the trajectories of $X$. This topic is discussed in detail in the paper \cite{LeiNualart} by Nualart and Lei. On the other hand, \textbf{(H2)} is an additional condition that we require for being able to handle ad-hoc computations related to the sequential compactness of $Z_{F}^{(n)}$. Moreover, under both assumptions   the convergence in probability, under  the topology of uniform convergence over compact sets, of the empirical spectral measure processes $\mu^{(n)} := (\mu_t^{(n)};t\geq0)$ holds, see Jaramillo et al.~\cite{jaramillo2018convergence}. That being said, one should observe that it is natural to expect \textbf{(H1)} and \textbf{(H2)} to be suboptimal conditions, as they are required by a particular technique we have used, rather than an intrinsic feature of the problem. We must mention however, that in practice, the level of generality that we present here encompasses  virtually all typically used models in Gaussian processes.

In order to state our main result, which is a functional central limit theorem for $Z_F^{(n)}$, we recall that the total variation distance  between two probability measures $\mu$ and $\nu$ is defined as follows
\begin{align*}
d_{TV}(\mu,\nu)
  &:=\sup_{A\in\mathcal{B}(\R)}\big|\mu(A)-\nu(A)\big|,
\end{align*}
where $\mathcal{B}(\R)$ denotes the Borel $\sigma$-algebra of $\R$. 

\begin{Theorem}\label{thm:main} 
Suppose that the processes $\{X_{i,j}; i\le j  \}$ satisfy  conditions \textbf{ (H1)} and \textbf{ (H2)}. Then, for every $F:=(f_{1},\dots, f_{r})\in \mathcal{P}^r$,  there exists a continuous $\R^{r}$-valued centered Gaussian process $\Lambda_{F}=((\Lambda_{f_{1}}(t),\dots, \Lambda_{f_{r}}(t)); t\geq0)$, independent of $\{X_{i,j}; i\le j\}$, defined on an extended probability space $(\Omega, \Gc,\Pb)$, such that
\begin{align}\label{eq:stableconvmain}
Z_F^{(n)}\stackrel{\mathcal{S}}{\longrightarrow}\Lambda_{F},
\end{align}
 in the topology of uniform convergence over compact sets. The law of the process $\Lambda_{F}$ is characterized by its covariance function, which is given by
\begin{align*}
\E\left[\Lambda_{f_i}(s)\Lambda_{f_j}(t)\right]      
  &=2\int_{\R^{2}}f_i^{\prime}(x)f_j^{\prime}(y)\nu_{\sigma_s,\sigma_t}^{\rho_{s,t}}(\ud x,\ud y),
\end{align*}
where $\nu_{\sigma_s,\sigma_t}^{\rho_{s,t}}$ is given as in Theorem \ref{theorem:covariance}. Moreover, for all $t\geq0$ and $f\in\mathcal{P}$, there exists a constant $C>0$ that only depends on $t,f$ and the law of $X$, such that 
\begin{align*}
d_{TV}(\mathcal{L}(Z_f^{(n)}(t)),\mathcal{L}(\Lambda_{f}(t)))
  &\leq \frac{C}{\sqrt{n}},
\end{align*}
where $\mathcal{L}(Z_f^{(n)}(t))$ and $\mathcal{L}(\Lambda_{f}(t))$ denote the distributions of $Z_f^{(n)}(t)$ and $\Lambda_{f}(t)$, respectively.
\end{Theorem}

We point out that  Theorem \ref{thm:main} is stated in terms of stable convergence instead of convergence in law to emphasize the fact that, as $n$ goes to infinity, $Z_F^{(n)}$ becomes asymptotically independent to any fixed event in $\Fc$, namely,
\begin{align*}
\lim_{n\rightarrow\infty}\E[\psi(Z_F^{(n)})\Indi{A}]
  &=\E[\psi(\Lambda_F)]\Pb[A],
\end{align*}
for every $A\in\mathcal{F}$ and every real-valued continuous functional $\psi$, with respect to the topology of uniform convergence over compact sets. In the spirit of Remark \ref{ref:rem1}, we would like to emphasize the potential of the stable convergence \eqref{eq:stableconvmain} to be applied in future research as a tool for determining non-central limit theorems for matrix-valued processes, in analogy with papers like \cite{CorNuWoe}, \cite{BinNourNualart} and \cite{HaJaNualart}.

We also note that when $H=1/2$ in \eqref{BMfrac} , the processes $X_{i,j}$ are Brownian motions. Hence, Theorem \ref{thm:main} can be thought of as a GOE version of Theorem~4.3 in \cite{perez2007functional}. However, the results in \cite{perez2007functional} holds only when the dimension $r$ equals one,
$f$ is a polynomial, and moreover the form of the limiting distribution is not explicit. On the other hand, Theorem \ref{thm:main} holds for all $r\in\N$ and we only require $f$ to satisfy a polynomial growth condition. In addition, the limiting distribution that we obtain is explicit.

\noindent To prove Theorem \ref{thm:main} we need to establish the convergence of the finite dimensional distributions of $Z_F^{(n)}$, as well as  the sequential compactness of $Z_F^{(n)}$ with respect to the topology of uniform convergence over compact sets, property that in the sequel will be referred to as ``tightness property''. These problems will be addressed in Sections \ref{sec:fdd} and \ref{sec:tight} respectively. The proof of the finite dimensional distributions relies on a multivariate central limit theorem, first presented in \cite{nourdin2009second} by Nourdin, Peccati and R\'eveillac. This central limit theorem is part of a series of very powerful techniques that provides convergence to Gaussian laws, and combine Malliavin calculus and Stein's method techniques. We refer the reader to \cite{nourdin2012normal} for a comprehensive presentation of these type of results.  On the other hand, due to the generality of the covariance function $R(s,t)$, the proof of the tightness property for $Z_F^{(n)}$ is a challenging problem since  Billingsley's criterion (see Theorem~12.3 in \cite{billingsley2013convergence}), a typical tool for proving tightness, requires us to compute moments of large order for the increments of $Z_F^{(n)}$. To overcome this difficulty, we use the results from \cite{jaramillo2018convergence}, to write a Skorohod differential equation for $Z_F^{(n)}$ of the type 
\begin{align}\label{eq:Zfstochprelim}
Z_F^{(n)}(t)
  &=\delta^{*}(\Indi{[0,t]}(\cdot)h(Y^{(n)}(\cdot)))+\int_{0}^{t}g(Y^{(n)}(\cdot))\ud s,
\end{align}
for some functions $h,g:\R^{n\times n}\rightarrow\R$ depending on $n$ and  where $\delta^{*}$ denotes the extended divergence (see Section \ref{sec:chaos} for a proper definition). Then we use Malliavin calculus techniques to estimate 
\begin{align}\label{eq:moments}
\E\Big[\Big|Z_F^{(n)}(t)-Z_F^{(n)}(s)\Big|^p\Big]
\end{align}
for $t>s$ and $p\geq 2$ even, which gives the tightness property. Although the Malliavin calculus perspective for proving tightness has already been explored in previous papers, see for instance Jaramillo and Nualart \cite{jaramillo2017functional} and  Harnett et al.~\cite{harnett2017symmetric}, its combination with a representation of the type \eqref{eq:Zfstochprelim} for estimating the moments \eqref{eq:moments} is a new ingredient that we have incorporated to our proof, and that seems to be quite effective in the context of matrix-valued processes.\\

\noindent The remainder of this paper is organized as follows. In Section \ref{sec:chaos}, we present some preliminaries on classical Malliavin calculus, random matrices and free Wigner integrals. Section \ref{sec:covariance} is devoted to the proof of Theorem \ref{theorem:covariance}. In Section \ref{sec:fdd} we prove the convergence of the finite dimensional distributions of $Z_F^{(n)}$ and finally, in Section \ref{sec:tight}, we prove the tightness property for $Z_F^{(n)}$.
\section{Preliminaries on Malliavin calculus and stochastic integration}\label{sec:chaos}
\subsection{Malliavin calculus for classical Gaussian processes}\label{subsec:chaos}

In this section, we establish some notation and introduce the basic operators of the theory of Malliavin calculus. Unless indicated otherwise, the material presented in this section can be found in the monographs of   Nourdin and Pecatti \cite{nourdin2012normal} and Nualart \cite{nualart2006malliavin}. Throughout this section  $X=(X_{t} , t\geq0)$ denotes a $d$-dimensional centered Gaussian process where $X_{t}=(X_{t}^{1},\dots, X_{t}^{d})$ for $t\ge0$, which is  defined on a probability space $(\Omega,\Fc,\Pb)$. Its covariance function is given by
\begin{align*}
\E\left[X_{s}^{i}X_{t}^{j}\right]
  &=\delta_{i,j}R(s,t),
\end{align*}
for some non-negative definite function $R(s,t)$ satisfying conditions \textbf{(H1)} and \textbf{(H2)} and where $\delta_{i,j}$ denotes the so-called Kronecker delta. 
We denote by $\Hg$ the Hilbert space obtained by taking the completion of the space $\Ef$ of step functions over $[0,T]$, endowed with the inner product 
\begin{align*}
\Ip{\Indi{[0,s]},\Indi{[0,t]}}_{\Hg}
  &:=  \E\left[X_{s}^{1}X_{t}^{1}\right],\qquad \text{ for }\quad 0\leq s,t\leq T.
\end{align*}
For every $1\leq j\leq d$ fixed, the mapping $\Indi{[0,t]} \mapsto X_{t}^{j}$ can be extended to a linear isometry between $\Hg$ and the closed Gaussian subspace of ${\rm L}^{2}\left(\Omega\right)$ generated by the process $X^{j}$. We  denote this isometry by $X^{j}(h)$, for $h\in\Hg$. If $h\in \Hg^{d}$ then is of the form $h=(h_{1},\dots, h_{d})$, with $h_{j}\in \Hg$, and we set $X(h):= \sum_{j=1} ^d X^{j}(h_{j})$.  Then $h\mapsto X(h)$ is a linear isometry between $\Hg^{d}$ and the closed Gaussian subspace of ${\rm L}^2\left(\Omega \right)$ generated by $X$.
 
For any integer $q\geq1$, we denote by $(\Hg^{d})^{\otimes q}$ and $(\Hg^{d})^{\odot q}$ the $q$-th tensor product of $\Hg^{d}$, and the $q$-th symmetric tensor product of $\Hg^{d}$, respectively. The $q$-th Wiener chaos	of ${\rm L}^{2}(\Omega)$, denoted by $\Hc_{q}$, is the closed subspace of ${\rm L}^{2}(\Omega)$ generated by the variables 
\[
\left (\prod_{j=1}^{d}H_{q_{j}}(X^j(v_{j}))\ \Big|\ \sum_{j=1}^{d}q_{j}=q,\text{ and } v_{1},\dots, v_{d}\in\Hg,\Norm{v_{j}}_{\Hg}=1 \right),
\]
 where $H_{q}$ is the $q$-th Hermite polynomial, defined by 
\begin{align*}
H_{q}(x)
  &:=(-1)^{q}e^{\frac{x^{2}}{2}}\frac{\text{d}^{q}}{\text{d}x^{q}}e^{-\frac{x^{2}}{2}}.
\end{align*}
For $q\in\N$, with $q\geq1$, and $h\in \Hg^{d}$ of the form $h=(h_{1},\dots, h_{d})$, with $\Norm{h_{j}}_{\Hg}=1$, we can write 
\begin{align*}
h^{\otimes q}
  =\sum_{i_{1},\dots, i_{q}=1}^{d}\hat{h}_{i_{1}}\otimes\cdots \otimes \hat{h}_{i_{q}},
\end{align*}
where $\hat{h}_{i}=(\underbrace{0,\dots, 0}_{i-1\text{\ times}},h_{i},\underbrace{0,\dots, 0}_{d-i \text{\ times}})$. For such $h$, we define the mapping 
\begin{align*}
I_{q}(h^{\otimes q})
  &:=\sum_{i_{1},\dots, i_{q}=1}^{d}\prod_{j=1}^{d}H_{q_{j}(i_{1},\dots, i_{q})}(X^j(h_{j})),
\end{align*}
where $q_{j}(i_{1},\dots, i_{q})$ denotes the number of indices in $(i_{1},\dots, i_{q})$ equal to $j$. The range of $I_{q}$ is contained in $\Hc_{q}$. Furthermore, this mapping can be extended to a linear isometry between $(\Hg^{d})^{\odot q}$ (equipped with the norm $\sqrt{q!}\Norm{\cdot}_{(\Hg^{d})^{\otimes q}}$) and $\Hc_{q}$ (equipped with the ${\rm L}^{2}(\Omega)$-norm). Such an extension is known as the multiple It\^{o} integral of order $q$ and, by abuse of notation, we denote it by $I_{q}$.

 Denote by $\Fc$ the $\sigma$-algebra generated by $X$. By the celebrated chaos decomposition theorem, every element $F \in{\rm L}^{2}(\Omega,\Fc)$  can be written as follows
\begin{align*}
F=\E\left[F\right]+\sum_{q=1}^{\infty}I_{q}(h_{q}),
\end{align*}
for some $h_{q}\in(\Hg^{d})^{\odot q}$. In what follows, for every integer $q\geq1$, we  denote by 
$$J_{q}:{\rm L}^{2}(\Omega,\Fc)\rightarrow {\rm L}^{2}(\Omega,\Fc),$$ 
the projection over the $q$-th Wiener chaos $\Hc_{q}$. 
Let $\Sf$ denote the set  of all cylindrical random variables of the form
\begin{align*}
F= g(X(h_{1}),\dots, X(h_{n})),
\end{align*} 
where $h_{j}\in\mathfrak{H}^{d}$ and $g:\R^{n}\rightarrow\R$ is an infinitely differentiable function such that $g$ and its partial derivatives have at most polynomial growth. In the sequel, we  refer to the elements of $\Sf$ as ``smooth random variables''. For every $r\geq1$, the Malliavin derivative of order $r$ of $F$ with respect to $X$, is the element of ${\rm L}^{2}(\Omega;(\Hg^d)^{\odot r})$ defined by 
\begin{align*}
D^rF
  &=\sum_{i_1,\dots, i_{r}=1}^{n}\frac{\partial^{r}g}{\partial x_{i_1}\cdots \partial x_{i_r}}(X(h_{1}),\dots, X(h_{n}))h_{i_1}\otimes\cdots\otimes h_{i_r}.
\end{align*}
For $p\geq1$ and $r\geq1$, the space $\D^{r,p}$ denotes the closure of $\Sf$ with respect to the norm $\Norm{\cdot}_{\D^{r,p}}$, defined by 
\begin{align}\label{eq:seminorm}
\Norm{F}_{\D^{r,p}}
  &:=\left(\E\left[\Abs{F}^{p}\right]+\sum_{i=1}^{r}\E\left[\Norm{D^{i}F}_{(\Hg^{d})^{\otimes i}}^{p}\right]\right)^{\frac{1}{p}}.
\end{align}
The operator $D^{r}$ can be extended to the space $\D^{r,p}$ by approximation with elements in $\Sf$. When we take $p=2$ in the seminorm \eqref{eq:seminorm}, we denote by $\delta$ the adjoint of the operator $D$, also called the divergence operator. We point out that every element $F\in \D^{1,2}$ satisfies  Poincar\'e's inequality 
\begin{align}\label{eq:poincare}
\text{Var}[F]
  &\leq \E[\|DF\|_{\Hg^d}^2],
\end{align}
where ${\rm Var}[F]$ denotes the variance of $F$ under $\mathbb{P}$.

Let ${\rm L}^{2}(\Omega;\Hg^{d})$ denote the space of square integrable random variables with values in $\Hg^{d}$. A random element $u\in {\rm L}^{2}(\Omega;\Hg^{d})$ belongs to the domain of $\delta$, denoted by $\mathrm{Dom} \, \delta $, if and only if it satisfies
\begin{align*}
\Abs{\E\left[\Ip{DF,u}_{\Hg^{d}}\right]}
  &\leq C_{u}\E\left[F^{2}\right]^{\frac{1}{2}},\ \text{ for every } F\in\D^{1,2},
\end{align*}
where $C_{u}$ is  a constant only depending on $u$. If $u\in \mathrm{Dom} \,\delta$, then the random variable $\delta(u)$ is defined by the duality relationship
\begin{align}\label{eq:dualitydeltaD}
\E\left[F\delta(u)\right]=\E\left[\Ip{DF,u}_{\Hg^{d}}\right],
\end{align}
which holds for every $F\in\D^{1,2}$.\\

Next we present a brief discussion regarding the connection between the divergence $\delta$ and the notion of stochastic integral. The reader should keep in mind that although this connection is discussed in most of the surveys on Malliavin calculus, the particular point of view that we take is perhaps specialized, so we rather recommend the paper \cite{LeiNualart} as the main reference for the remainder of this section. 
If $X$ is a $d$-dimensional Brownian motion, thus $R(s,t)=s\wedge t$ and $\Hg={\rm L}^{2}[0,T]$. In this case,  the operator $\delta$
 is an extension of the It\^o integral. Motivated by this fact, if $u$ is a random variable with values in $({\rm L}^{p}[0,T])^{d}\cap \Hg^{d}$, for some $p\geq 1$, we would like to interpret $\delta(u)$ as a stochastic integral. 
Nevertheless, the space $\Hg$ turns out to be too small for this purpose, as generally it  doesn't contain important elements $u\in({\rm L}^{p}[0,T])^{d}$, for which we would like $\delta(u)$ to make sense. To be precise,  in \cite{cheridito2005stochastic} it was shown that in the case where $X$ is a fractional Brownian motion with Hurst parameter $0<H<\frac{1}{4}$, and covariance function 
$$R(s,t)=\frac{1}{2}(t^{2H}+s^{2H}-\Abs{t-s}^{2H}),$$ 
the trajectories of $X$ do not belong to the space $\Hg$, and in particular, non-trivial processes of the form $(h(u_{s}),  s\in[0,T])$, with $h:\R\rightarrow\R$, do not belong to the domain of $\delta$. In order to overcome this difficulty, we extend the domain of $\delta$ by following the approach presented in \cite{lei2012stochastic} (see also \cite{cheridito2005stochastic}). The main idea consists on extending the definition of $\Ip{\varphi,\psi}_{\Hg}$ to the case where $\varphi\in {\rm L}^{\frac{\alpha}{\alpha-1}}[0,T]$ for some $\alpha>1$, and $\psi$ belongs to the space $\Ef$ of step functions over $[0,T]$.

 Let $\alpha>1$ be as in hypothesis \textbf{(H1)} and let $\bar{\alpha}$ be the conjugate of $\alpha$, defined by $\bar{\alpha}:=\frac{\alpha}{\alpha-1}$. For any pair of functions $\varphi\in {\rm L}^{\bar{\alpha}	}([0,T];\R)$ and $\psi\in\Ef$ of the form $\psi=\sum_{j=1}^{m}c_{j}\Indi{[0,t_{j}]}$, we define
\begin{align}\label{def:extendedIP}
\Ip{\varphi,\psi}_{\Hg}
  &:=\sum_{j=1}^{m}c_j\int_{0}^{T}\varphi(s)\frac{\partial R}{\partial s}(s,t_{j})\ud s.
\end{align}
This expression is well defined since
\begin{align*}
\Abs{\Ip{\varphi, \Indi{[0,t]}}_{\Hg}}
  &=\Abs{\int_{0}^{T}\varphi_{s}\frac{\partial R}{\partial s}(s,t)\ud s}
	\leq\Norm{\varphi}_{{\rm L}^{\bar{\alpha}}[0,T]}\sup_{0\leq t\leq T}\left(\int_{0}^{T}\Abs{\frac{\partial R}{\partial s}(s,t)}^{\alpha}\ud s\right)^{\frac{1}{\alpha}}<\infty.
\end{align*}
One should keep in mind that the notation used in definition \eqref{def:extendedIP}, is the same one that we use to describe the inner product of $\Hg$. This abuse of notation is justified by the fact that the bilinear function \eqref{def:extendedIP} coincides with the inner product in $\Hg$, when $\varphi\in\Ef$. Indeed, for $\varphi\in\Ef$ of the form $\varphi=\sum_{i=1}^{n}a_{i}\Indi{[0,t_{i}]},$ we have
\begin{align*}
\Ip{\varphi,\Indi{[0,t]}}_{\Hg}
  &=\sum_{i=1}^{n}a_{i}R(t_{i},t)
	=\sum_{i=1}^{n}a_{i}\int_{0}^{t_{i}}\frac{\partial R}{\partial s}(s,t)\ud s=\int_{0}^{T}\varphi(s)\frac{\partial R}{\partial s}(s,t)\ud s.
\end{align*}
We define the extended domain of the divergence as follows.
\begin{Definition}\label{def:extendeddelta}
Let $\Ip{\cdot,\cdot}_{\Hg}$ be the bilinear function defined by \eqref{def:extendedIP}. We say that a stochastic process $u\in {\rm L}^{1}(\Omega;{\rm L}^{\bar{\alpha}}([0,T];\R^{d}))$ belongs to the extended domain of the divergence, denoted by $\mathrm{Dom}\, \delta^{*}$, if there exists $\gamma>1$ such that 
\begin{align*}
\Abs{\E\left[\Ip{DF,u}_{\Hg^{d}}\right]}
  &\leq C_{u}\Norm{F}_{{\rm L}^\gamma(\Omega)},
\end{align*}
for any smooth random variable $F\in\Sf$, where $C_{u}$ is some constant depending on $u$. In this case, $\delta^*(u)$ is defined by the duality relationship
\begin{align*}
\E\left[F\delta^*(u)\right]=\E\left[\Ip{DF,u}_{\Hg^{d}}\right].
\end{align*}
\end{Definition}
It is important to note that for a general covariance function $R(s,t)$ and $\beta>1$, the domains $\mathrm{Dom}^{*}\,\delta$ and $\mathrm{Dom}\, \delta$ are not necessarily comparable (see Section 3 in \cite{lei2012stochastic} for further details). We also note that along the paper we use of the notation 
\begin{align}\label{eq:Skorohod_def}
\sum_{i=1}^d\int_{0}^{t} u_{s}^{i}\delta X_{s}^{i}
  &:=\delta^{*}(u\Indi{[0,t]}),
\end{align} 
for $u\in \mathrm{Dom}\,\delta^{*}$ of the form $u_t=(u^{1}_{t},\dots, u_{t}^{d})$.\\

\noindent Next, we introduce the operator $\mathcal{L}$ which is an unbounded linear mapping, defined in a suitable subdomain of ${\rm L}^{2}(\Omega,\Fc)$, taking values in ${\rm L}^{2}(\Omega,\Fc)$ and given by the formula
$$\mathcal{L}F:=\sum_{q=1}^{\infty}-q J_{q}F.$$
Moreover, the operator $\mathcal{L}$ coincides with the infinitesimal generator of the Ornstein-Uhlenbeck semigroup $(P_{\theta}, \theta\geq0)$, which is defined as follows
\begin{align*}
\begin{array}{cccc}P_{\theta}: & {\rm L}^{2}(\Omega,\Fc)&\rightarrow & {\rm L}^{2}(\Omega,\Fc)\\
                               & F                &\mapsto     & \sum_{q=0}^{\infty}e^{-q\theta}J_{q}F.\end{array}
\end{align*}
We also observe that a random variable $F$ belongs to the domain of $\mathcal{L}$ if and only if $F\in \D^{1,2}$, and $DF\in \mathrm{Dom} \, \delta$, in which case
\begin{align}\label{eq:deltaDFI}
\delta DF
  &=-\mathcal{L}F.
\end{align}
We also define the operator $\mathcal{L}^{-1}:{\rm L}^{2}(\Omega,\Fc)\rightarrow {\rm L}^{2}(\Omega,\Fc)$ by
$$\mathcal{L}^{-1}F=\sum_{q=1}^{\infty}-\frac{1}{q}J_{q}F.$$
Notice that $\mathcal{L}^{-1}$ is a bounded operator and satisfies $\mathcal{L}\mathcal{L}^{-1}F=F-\E\left[F\right]$ for every $F\in {\rm L}^{2}(\Omega)$, so that $\mathcal{L}^{-1}$ acts as a pseudo-inverse of $\mathcal{L}$. The operator $\mathcal{L}^{-1}$ satisfies the following contraction property for every $F\in {\rm L}^{2}(\Omega)$ with $\E\left[F\right]=0$, 
\begin{align*}
\E\left[\Norm{D\mathcal{L}^{-1}F}_{\Hg^{d}}^2\right]
  &\leq\E\left[F^2\right].
\end{align*}
In addition, by Meyer's inequalities (see for instance Proposition~1.5.8 in \cite{nualart2006malliavin}), for every $p>1$ there exists a constant $c_{p}>0$ such that for all $F\in \D^{2,p}$ with $\E\left[F\right]=0$,
\begin{align}\label{eq:Meyer}
\E[|\delta(D\mathcal{L}^{-1}F)|^p]^{\frac{1}{p}}
  &\leq c_{p}\Big(\E[\|D^2\mathcal{L}^{-1}F\|_{(\Hg^d)^{\otimes 2}}^p]^{\frac{1}{p}}+\Norm{\E\left[D\mathcal{L}^{-1}F\right]}_{\Hg^d}\Big).
\end{align}
Assume that $\widetilde{X}$ is an independent copy of $X$, such that both r.v.'s are defined in the product space $(\Omega\times \widetilde{\Omega}, \Fc \otimes\widetilde{\Fc},\Pb\otimes\widetilde{\Pb})$. Given a random variable $F\in {\rm L}^{2}(\Omega,\Fc)$, we can write $F=\Psi_{F}(X)$, where $\Psi_{F}$ is a measurable mapping from $\R^{\Hg^{d}}$ to $\R$, determined $\Pb$-a.s. Then, for every $\theta\geq0$ we have  Mehler's formula 
\begin{align}\label{eq:Mehler}
P_{\theta}F
  &=\widetilde{\E}\left[\Psi_{F}(e^{-\theta}X+\sqrt{1-e^{-2\theta}}\widetilde{X})\right],
\end{align}
where $\widetilde{\E}$ denotes the expectation with respect to $\widetilde{\Pb}$. The operator $-\mathcal{L}^{-1}$ can be expressed in terms of $P_{\theta}$, as follows
\begin{align}\label{eq:Mehler2}
-\mathcal{L}^{-1}F
  &=\int_{0}^{\infty}P_{\theta}F\ud\theta,\qquad \text{ for } F\,\,\text{  s.t.}\quad 
  \E\left[F\right]=0.
\end{align}
 Formulas \eqref{eq:deltaDFI}, \eqref{eq:Mehler} and \eqref{eq:Mehler2}, combined with Meyer's inequality \eqref{eq:Meyer}, allows us to write the ${\rm L}^{p}(\Omega)$-norm of any $F\in\D^{1,2}$, in the form
\begin{align*}
\left\|F-\E[F]\right\|_{{\rm L}^p(\Omega)}
  &=\left\|-\delta D\mathcal{L}^{-1}(F-\E[F])\right\|_{{\rm L}^p(\Omega)}\\
	&\leq C_p\left(\left\|\int_{0}^{\infty}DP_{\theta}[F]\ud\theta\right\|_{{\rm L}^p(\Omega;\Hg^{d})}
	+\left\|\int_{0}^{\infty}D^2P_{\theta}[F]\ud\theta\right\|_{{\rm L}^p(\Omega;(\Hg^{d})^{\otimes 2})}\right)\\
	&\leq C_p\left(\left\|\int_{0}^{\infty}e^{-\theta}P_{\theta}[DF]\ud\theta\right\|_{{\rm L}^p(\Omega;\Hg^{d})}
	+\left\|\int_{0}^{\infty}e^{-2\theta} P_{\theta}[D^2F]\ud\theta\right\|_{{\rm L}^p(\Omega;(\Hg^{d})^{\otimes 2})}\right),
\end{align*}
where $C_{p}>0$ is a universal constant only depending on $p$. Thus, using Minkowski's inequality and the contraction property of $P_{\theta}$ with respect to ${\rm L}^{p}(\Omega)$, we have that
\begin{align}\label{eq:centeredLpbound}
\left\|F-\E[F]\right\|_{{\rm L}^p(\Omega)}
  &\leq C_p\int_{0}^{\infty}e^{-\theta}\Big(\|P_{\theta}[DF]\|_{{\rm L}^p(\Omega;\Hg^{d})}+\|P_{\theta}[D^2F]\|_{{\rm L}^p(\Omega;(\Hg^{d})^{\otimes 2})}\Big)\ud\theta\nonumber\\
	&\leq C_p\int_{0}^{\infty}e^{-\theta}\Big(\|DF\|_{{\rm L}^p(\Omega;\Hg^{d})}+\|D^2F\|_{{\rm L}^p(\Omega;(\Hg^{d})^{\otimes 2})}\Big)\ud\theta\nonumber\\
	& =C_p \left(\|DF\|_{{\rm L}^p(\Omega;\Hg^{d})}+\|D^2F\|_{{\rm L}^p(\Omega;(\Hg^{d})^{\otimes 2})}\right).
\end{align}


Finally, we recall the notion of the contraction in $\Hg^{d}$. Let $\{b_j, j\ge 1 \}\subset \Hg^{d}$ be a complete orthonormal system of $\Hg^{d}$. Given $f\in(\Hg^{d})^{\odot p}$, $g\in(\Hg^{d})^{\odot q}$ and $r\in\{1,\dots, p\wedge q\}$, the $r$-th contraction of $f$ and $g$ is the element $f\otimes_{r}g\in (\Hg^{d})^{\otimes(p+q-2r)}$ defined by 
\begin{align*}
f\otimes_{r}g
  &=\sum_{i_1,\dots, i_{r}=1}^{\infty}\langle f,b_{i_1},\dots, b_{i_r}\rangle_{(\Hg^{d})^{\otimes r}}\otimes \langle g,b_{i_1},\dots, b_{i_r}\rangle_{(\Hg^{d})^{\otimes r}}.
\end{align*}
\subsection{Central limit theorem in the Wiener chaos}
The proof of the stable convergence of the finite dimensional distributions of $Z_{F}^{(n)}$ in Theorem \ref{thm:main}, is based on Theorem \ref{thm:CLTWienerchaos} below, which is a combination of the paper \cite{nourdin2010second} by Nourdin, Peccati and R\'eveillac and the paper \cite{nourdin2009second} by Nourdin, Peccati and Reinert. The proofs of these results can be found in the monograph of Nourdin and Pecatti \cite{nourdin2012normal} (see Theorems~5.3.3~and~6.1.3).
\begin{Theorem}\label{thm:CLTWienerchaos}
Fix $d\geq 1$ and consider the sequence of vectors $\{{\tt Z}_{n}=({\tt Z}_{n}^1,\dots, {\tt Z}_{n}^d)$, $n\geq 1$\}, with $\E\left[{\tt Z}_{n}^i\right]=0$ and ${\tt Z}_{n}^i\in\D^{2,4}$ for every $i\in\{1,\dots, d\}$ and $n\geq 1$. Let $N=(N_1,\dots, N_d)$ be a centered Gaussian vector with covariance $C$ which is  a symmetric and non-negative square matrix of dimension $d$. If the following conditions are fulfilled
\begin{enumerate}
\item[(i)] for any $i,j\in\{1,\dots, d\}$, we have $\E\left[{\tt Z}_{n}^i{\tt Z}_{n}^j\right]\rightarrow C(i,j)$, as $n\rightarrow\infty$;
\item[(ii)] for any $i\in\{1,\dots, d\}$, we have $\sup_{n\geq 1}\E\left[\Norm{D{\tt Z}_{n}^i}_{\Hg}^{4}\right]<\infty$ and 
\item[(iii)] for any $i\in\{1,\dots, d\}$, we have $\E\left[\Norm{D^2{\tt Z}_{n}^i\otimes_{1}D^2{\tt Z}_{n}^i}_{(\Hg^d)^{\otimes 2}}^2\right]\rightarrow 0$, as $n\rightarrow\infty$,
\end{enumerate}
then ${\tt Z}_{n}\xrightarrow[]{(law)}\mathcal{N}_{d}(0,C)$, as $n\rightarrow\infty$, and moreover
\begin{align*}
d_{TV}(\mathcal{L}({\tt Z}_{n}^i),\mathcal{L}(N_{i}))
  &\leq C_1 \E\left[\Norm{D^2{\tt Z}_{n}^i\otimes_{1}D^2{\tt Z}_{n}^i}_{(\Hg^d)^{\otimes 2}}^2\right]^{\frac{1}{4}},
\end{align*}
where $C_1>0$ is a constant independent of $n$ and $\mathcal{L}({\tt Z}_{n}^i)$ and $\mathcal{L}(N_i)$ denote the laws of ${\tt Z}_{n}^i$ and $N_i$, respectively. 
\end{Theorem}
Theorem \ref{thm:CLTWienerchaos} is closely related to the celebrated {\it Fourth Moment Theorem}, originally established  by Nualart and Peccati in \cite{nualart2005central} where convergence in distribution of multiple Wiener integrals to the standard Gaussian law is stated, in the sense that it is equivalent to the convergence of just the fourth moment. For further details about  improvements and developments  on this subject we refer to the monograph of Nourdin and Pecatti \cite{nourdin2012normal}.
\subsection{Free independence and multiple Wigner integrals}\label{sec:freemultint}
The proof of Theorem \ref{theorem:covariance}, uses the relation between classical independence of large symmetric random matrices and free independence of non-commutative random variables, which was first explored by Voiculescu in \cite{voiculescu1991limit}. In this section we introduce some basic tools from free probability regarding analysis in the Wigner space, which are very useful for our purposes. We  closely follow Biane and Speicher \cite{biane1998stochastic} and  Kemp et al.~\cite{kemp2012wigner}.

A $C^*$-probability space is a pair $(\mathcal{A},\tau)$ where $\mathcal{A}$ is a unital $C^*$-algebra  and $\tau:\mathcal{A}\to\mathbb{C}$ is a positive unital linear functional. In the sequel, the involution associated to $\mathcal{A}$ will be denoted by $*$. Two classical examples to keep in mind are the following, 
 \begin{enumerate}
 \item[(i)] the algebra $\mathcal{A}$ of bounded $\mathbb{C}$-valued random variables defined in a given probability space, where $\tau=\E[\cdot]$ is the expectation and $*$ denotes the complex conjugation and 
\item[(ii)] the algebra $\mathcal{A}$ of random matrices of dimension $n$, where $\tau$ is the expected normalized trace $\frac{1}{n}\E[{\rm Tr}(\cdot)]$ and $*$ denotes the conjugate transpose operation. 
  \end{enumerate}
  
 The elements of $\mathcal{A}$ are called non-commutative random variables. In the sequel we  use the symbol $*$ to denote, both, the involution of a $C^*$-probability space (when applied to a non-commutative random variable) and the conjugate transpose operation (when applied to a matrix). This abuse of notation is justified by the fact that, as mentioned in example (ii), the set of random matrices of dimension $n$ can be realized as a $C^*$-probability space.
 
 An element $a\in\mathcal{A}$ such that $a=a^*$ is called self-adjoint. A $W^{*}$-probability space is a $C^*$-probability space $(\mathcal{A},\tau)$ such that $\mathcal{A}$ is a Von Neumann algebra (i.e., an algebra of operators on a separable Hilbert space, closed under adjoint and weak convergence) and $\tau$ is weakly continuous, faithful (i.e., that if $\tau[YY^*]=0$, then $Y=0$) and tracial (i.e., that $\tau[XY]=\tau[YX]$ for all $X,Y\in\mathcal{A}$). The functional $\tau$ should be understood as the analogue of the expectation in classical probability. For $a_1,\dots,a_k\in \mathcal{A}$, we  refer to the values of $\tau[a_{i_1}\cdots a_{i_n}]$, for $1\leq i_1,...,i_{n}\leq k$ and  $n\geq1$, as the \textit{mixed moments} of $a_1,\dots,a_k$.

 For any self-adjoint element $a\in\mathcal{A}$, there exists a unique probability measure $\mu_a$ supported over a compact subset of the reals numbers such that
 $$\int_{\mathbb{R}}x^{k}\mu_a ({\rm d}x)=\tau [a^{k}], \quad \textrm{for }\quad k\in \mathbb{N}.$$
 The measure $\mu_a$ is often called the (analytical) distribution of $a$.

 Even if we know the individual distribution of two self-adjoint elements $a,b\in \mathcal{A}$, their joint distribution (mixed moments) can be quite arbitrary, unless some notion of independence is assumed to hold between $a$ and $b$.  Here, we  deal  with free independence.

\begin{definition} Let $\{A_i, i\in\iota\}$ be a family of subalgebras of $\mathcal{A}$ and, for $a\in \mathcal{A}$, let $\mathring{a}:=a-\tau[a]$. 
We say that $\{A_i, i\in\iota\}$ are \textit{freely independent} or \text{free} if
\begin{equation}
\tau[\mathring{a}_1\mathring{a}_2 \cdots \mathring{a}_k]=0,
\end{equation}
whenever $k\geq 1$, $a_1,\dots a_k\in \mathcal{A}$ with $a_j\in A_{i(j)}$ for $1\leq j\leq k$, and $i(1)\neq i(2) \neq \cdots \neq i(k)$.
\end{definition}
We now introduce the notion of a {\it free Brownian motion}. Let $S=(S_{t}\ ,\ t\geq 0)$ be  a one-parameter family of self-adjoint operators $S_{t}$, defined in a $W^{*}$ probability space $(\mathcal{A},\tau)$ satisfying 
\begin{enumerate}
\item[i)] $S_{0}=0$, 
\item[ii)] for all $0<t_{1}<t_{2}$, the increment  $S_{t_{2}}-S_{t_{1}}$ possesses the same law as the semicircular law with mean zero and variance $t_{2}-t_{1}$,
\item[iii)] and for all $k$ and $t_{1}\leq t_{2}\leq \cdots\leq t_{k-1}\leq t_{k}$, the increments $S_{t_{1}},S_{t_{2}}-S_{t_{1}},\dots, S_{t_{k+1}}-S_{t_{k}}$ are freely independent.
\end{enumerate}
The family of self-adjoint operators $S$ is known as free Brownian motion.

Let $f\in {\rm L}^{2}(\R_{+}^{q})$ be an off-diagonal indicator function of the form 
$$f(x_{1},\dots, x_{q})
  =\Indi{[s_{1},t_{1}]}(x_1)\cdots \Indi{[s_{q},t_{q}]}(x_q),$$
where the intervals $[s_{1},t_{1}], \dots, [s_{q},t_{q}]$ are pairwise disjoint. The Wigner integral $I_{q}^{S}(f)$ is defined as 
\begin{align*}
I_{q}^{S}(f)
  &:=(S_{t_{1}}-S_{s_{1}})\cdots(S_{t_{q}}-S_{s_{q}}),
\end{align*}
and then extended linearly over the set of all off-diagonal step-functions, which is dense in ${\rm L}^{2}(\R_{+}^q)$. The Wigner integral satisfies the following relation
\begin{align}\label{eq:isometrywigner}
\tau\big[I_{q}^{S}(f)^*I_{q}^{S}(g)\big]
  &=\Ip{f,g}_{{\rm L}^{2}(\R_{+}^q)}.
\end{align}
Namely, $I^S_q$ is an isometry from the space of off-diagonal step functions into the Hilbert space of operators generated by $S$, equipped with the inner product $\Ip{X,Y}=\tau[Y^*X]$. 

As a consequence, $I^S_{q}$ can be extended to the domain ${\rm L}^{2}(\R_{+}^{q})$. The Wigner integral has the property  that the image if $I^S_{m}$ is orthogonal to $I^S_{n}$ for $n\neq m$. In the sequel, we  use the notation $S(h):=I_1^{S}(h)$, for every $h\in {\rm L}^{2}(\R_{+})$.

\begin{Definition}
Let $m,n\in \mathbb{N}$, $f\in {\rm L}^{2}(\R_{+}^{n})$ and $g\in {\rm L}^{2}(\R_{+}^{m})$. For $p\leq m\wedge n$, we define the $p$-th contraction $f\stackrel{p}{\frown}g$ of $f$ and $g$ as the ${\rm L}^{2}(\R_{+}^{n+m-2p})$ function defined by
\begin{align*}
f\stackrel{p}{\frown}g(t_{1},\dots, t_{n+m-2p})
  &=\int_{\R_{+}^{p}}f(t_{1},\dots, t_{n-p},s_{1},\dots, s_{p})\\
	&\hspace{3cm}\times g(s_{p},\dots, s_{1},t_{n-p+1},\dots, t_{n+m-2p})\ud s_1\cdots \ud s_{p}.
\end{align*}
\end{Definition}
The following result was proved in \cite{biane1998stochastic},
\begin{Proposition}
Let $n,m\in \mathbb{N}$,   $f\in {\rm L}^{2}(\R_{+}^{n})$ and $g\in {\rm L}^{2}(\R_{+}^{m})$. Then, 
\begin{align*}
I_n^{S}(f)I_m^{S}(g)
  &=\sum_{p=0}^{n\wedge m}I^S_{n+m-2p}(f\stackrel{p}{\frown}g).
\end{align*}
\end{Proposition}
In the particular case when $n=1$, $m\geq 2$, $\Norm{f}_{{\rm L}^{2}(\R)}=1$ and $g=f^{\otimes m}$, we get 
\begin{align*}
S(f)I_m^{S}(f^{\otimes m})
  &=I_{1}^S(f)I_m^{S}(f^{\otimes m})=I^S_{m+1}(f^{\otimes(m+1)})+I^S_{m-1}(f^{\otimes(m-1)}),
\end{align*}
under the convention that $I_0^S$ is the identity function defined over $\R$. As a consequence, we have the recursion 
\begin{align*}
I^S_{m+1}(f^{\otimes(m+1)})
  &=S(f)I_m^{S}(f^{\otimes m})-I^S_{m-1}(f^{\otimes(m-1)}),
\end{align*}
with initial condition $I^S_{0}(f^{\otimes 0}) = 1$ and $I^S_{1}(f)=S(f)$. Since Chebyshev polynomials of the second kind are defined by the previous recursion, we conclude that 
\begin{align}\label{eq:chebandIto}
I^S_{q}(f^{\otimes q})
  &=U_{q}(S(f)),
\end{align}
where $U_{q}$ denotes the $q$-th Chebyshev polynomial of second order in $[-2,2]$, given by \eqref{eq:Chebyshevdef}. Hence, using the orthogonality of $I_m^S$ and $I_n^S$, as well as \eqref{eq:isometrywigner}, we obtain the property
\begin{equation}
\label{eq:IdentityPol}
\tau\left[U_{m}(S(f))U_{n}(S(g))\right]
  =\delta_{m,n}\Ip{f,g}_{{\rm L}^{2}(\R_{+})}^{m},
\end{equation}
 where  $\delta_{m,n}$ denotes the Kronecker delta. The previous equality shows that if $a$ and $b$ are jointly semicircular with mean zero and unit variance, then
\begin{align}\label{eq:isometrysemicircular}
\tau\left[U_{m}(a)U_{n}(b)\right]
  &=\delta_{m,n}(\tau\left[a^*b\right])^{m}.
\end{align}
Indeed, this is achieved by taking $f=\Indi{[0,1]}$ and $g=\Indi{[1-\tau[a^*b],2-\tau[a^*b]]}$ in \eqref{eq:IdentityPol}, so that $(I_1(f),I_2(g))$ and $(a,b)$ are equal in distribution.

\subsection{Eigenvalues of symmetric matrices}\label{sec:eigenvalues}

Define $d(n):=n(n+1)/2$. In the sequel, we  identify the elements $x=(x_{k,h} \vert 1\leq k\leq h\leq n)\in\R^{d(n)}$, with the $n$-dimensional, square symmetric matrix given by
\begin{align*}
\widehat{x}:=\left(\begin{array}{ccccccc}
\sqrt{2}x_{1,1} & x_{1,2} & x_{1,3} & \cdots & x_{1,n}\\
x_{1,2} & \sqrt{2}x_{2,2} & x_{2,3} & \cdots & x_{2,n}\\
x_{1,3} & x_{2,3} & \sqrt{2}x_{3,3} & \cdots & x_{3,n}\\
\vdots & \vdots & \vdots &\ddots & \vdots\\
x_{1,n} & x_{2,n} & x_{3,n} & \cdots & \sqrt{2}x_{n,n}
\end{array}\right).
\end{align*}
For every $x\in\R^{d(n)}$, we  denote by $\Phi_i(x)$ for the $i$-th largest eigenvalue of $\widehat{x}$. By Lemma 2.5 in the monograph of Anderson et al.~\cite{anderson2010introduction}, there exists an open subset $G\subset\R^{d(n)}$, with $|G^{c}|=0$, such that for every $x\in G$, the matrix $\widehat{x}$ has a factorization of the form $\widehat{x}=U DU^{*}$, where $D$ is a diagonal matrix with entries $D_{i,i}=\Phi_{i}(x)$ such that $\Phi_{1}(x)>\dots>\Phi_{n}(x)$, $U$ is an orthogonal matrix with $U_{i,i}>0$ for all $i$, $U_{i,j}\neq0$ and all the minors of $U$ have non zero determinants. Furthermore, if $\mathcal{O}(n)$ denotes the orthogonal group of dimension $n$ and $D_n$ the set of diagonal matrices of dimension $n$, there exist differentiable mappings 
$T_1:G\rightarrow\mathcal{O}(n)$ and $T_{2}:G\rightarrow D_{n}$, such that $\widehat{x}=T_{1}(x)T_2(x)T_{1}(x)^*$ for all $\in G$. For $x\in G$, we  denote by $U(x)$ for the orthogonal matrix $U(x)=T_{1}(x)$. 
The aforementioned set $G$ is known in the literature as the collection of good matrices, and is widely used in the study of the differential properties of Gaussian ensembles. To exemplify the importance of good matrices, we refer the reader to \cite[Section  4.3.2]{anderson2010introduction}, in which this set is utilized for setting the mathematical foundations for establishing a dynamical version of Wigner's theorem (which can be understood as a simplified first order version of the results presented here).

Let us  denote by $\frac{\partial \Phi_{i}}{\partial x_{k,h}}(x)$ the partial derivatives of $\Phi_i$ with respect to the $(k,h)$-component of $\widehat{x}$. In Lemma \ref{lem:app} in the appendix, it is shown that
\begin{align}
\frac{\partial \Phi_{i}}{\partial x_{k,h}}(x)
  =V_{k,h}^{i,i}(x),
	\label{eq:D1PhiUpc}
\end{align}
where
\begin{align}\label{eq:Vdef}
V_{k,h}^{i,j}(x)
  &:=\big(U_{k,i}U_{h,j}+U_{h,i}U_{k,j}\big)(x)\Indi{\{k\neq h\}}
	+\sqrt{2}U_{k,i}(x)U_{k,j}(x)\Indi{\{k=h\}}.
\end{align}
Next we prove some useful properties of the terms $V_{k,h}^{i,j}(x)$. It is not difficult to deduce  that for every $1\leq i,j\leq n$ and $1\leq k\leq h\leq n$, we have that $V_{k,h}^{i,j}(x)=V_{h,k}^{i,j}(x)$, and in consequence, we have
\[
\begin{split}
\sum_{k\leq h}V_{k,h}^{i_1,j_1}(x)V_{k,h}^{i_2,j_2}(x)
  &=\frac{1}{2}\sum_{k<h}V_{k,h}^{i_1,j_1}(x)V_{k,h}^{i_2,j_2}(x)+
  \frac{1}{2}\sum_{k<h}V_{h,k}^{i_1,j_1}(x)V_{h,k}^{i_2,j_2}(x)\\
  &\hspace{7cm}+
  \sum_{k=1}^n V_{k,k}^{i_1,j_1}(x)V_{k,k}^{i_2,j_2}(x)\\
	&=\frac{1}{2}\sum_{p\neq q}V_{p,q}^{i_1,j_1}(x)V_{p,q}^{i_2,j_2}(x)+\sum_{p=1}^n V_{p,p}^{i_1,j_1}(x)V_{p,p}^{i_2,j_2}(x).
\end{split}
\]
From here we obtain
\[
\begin{split}
\sum_{k\leq h}V_{k,h}^{i_1,j_1}(x)V_{k,h}^{i_2,j_2}(x)
  &=\frac{1}{2}\sum_{p\neq q}\big(U_{p,i_1}U_{q,j_1}+U_{q,i_1}U_{p,j_1}\big)(x)\big(U_{p,i_2}U_{q,j_2}+U_{q,i_2}U_{p,j_2}\big)(x)\\
  &\hspace{6.5cm}+2\sum_{p=1}^n (U_{p,i_1}U_{p,j_1}U_{p,i_2}U_{p,j_2})(x)\\
	&=\frac{1}{2}\sum_{1\leq p,q\leq n}\big(U_{p,i_1}U_{q,j_1}+U_{q,i_1}U_{p,j_1}\big)(x)\big(U_{p,i_2}U_{q,j_2}+U_{q,i_2}U_{p,j_2}\big)(x).
\end{split}
\]
Consequently, by the orthogonality of the columns of $U(x)$,  we have
\begin{align}\label{eq:VandW}
\sum_{k\leq h}V_{k,h}^{i_1,j_1}(x)V_{k,h}^{i_2,j_2}(x)
  &=\delta_{i_1,i_2}\delta_{j_1,j_2}+\delta_{i_1,j_2}\delta_{j_1,i_2}.
\end{align}
where we recall that $\delta_{i,j}$ denotes the Kronecker delta.
Using identities  \eqref{eq:D1PhiUpc} and \eqref{eq:VandW}, we get that for every $1\leq i_{1},i_{2}\leq n$, 
\begin{align}
\sum_{k\leq h}\frac{\partial \Phi_{i_1}}{\partial x_{k,h}}(x)\frac{\partial \Phi_{i_2}}{\partial x_{k,h}}(x)
  &=2\Indi{\{i_1=i_{2}\}},\label{eq:D1PhiUpsum}
\end{align}
which in turn implies that for every function $f:\R\rightarrow\R$, and $x\in G$, the functionals
\begin{align}
\Psi_{k,h}[f](x)
  &:=\sum_{i=1}^{n}f(\Phi_{i}(x))\frac{\partial \Phi_{i}}{\partial x_{k,h}}(x),\label{eq:psik1def}\\
\Psi_{k,h}^{p,q}[f](x)
  &:=\sum_{i=1}^{n}f(\Phi_{i}(x))\frac{\partial \Phi_{i}}{\partial x_{k,h}}(x)\frac{\partial \Phi_{i}}{\partial x_{p,q}}(x)\label{eq:psik2def},
\end{align}
satisfy
\begin{align}
\sum_{k\leq h}\big|\Psi_{k,h}[f](x)\big|^2
	& =2\sum_{i=1}^{n}f(\Phi_{i}(x))^2,\label{eq:isimetrypsik1}\\
\sum_{k\leq h}\sum_{p\leq q}\big|\Psi_{k,h}^{p,q}[f](x)\big|^2
	& =4\sum_{i=1}^{n}f(\Phi_{i}(x))^2\label{eq:isimetrypsik2}.
\end{align}
On the other hand, from Lemma \ref{lem:app} (see Apendix) we know
\begin{align}
\frac{\partial^{2}\Phi_{i}}{\partial x_{k,h}\partial x_{p,q}}(x)
  &=\sum_{j=1}^n\frac{2}{\Phi_{i}(x)-\Phi_{j}(x)}\Indi{\{j\neq i\}}V_{k,h}^{i,j}(x)V_{p,q}^{i,j}(x).\label{eq:D2partialPhiUp}
\end{align}
Thus, we get that for every $k\leq h$, $p\leq q$, 
\begin{equation}\label{eq:Fpartialv01}
\begin{split}
\sum_{i=1}^{n}f(\Phi_{i}(x))\frac{\partial^{2} \Phi_{i}}{\partial x_{k,h}\partial x_{p,q}}(x)
	&=2\sum_{i\neq j}\frac{f(\Phi_{i}(x))}{\Phi_{i}(x)-\Phi_{j}(x)}V_{k,h}^{i,j}(x)V_{p,q}^{i,j}(x)\\
	&\hspace{-3cm}=\sum_{i\neq j}\frac{f(\Phi_{i}(x))}{\Phi_{i}(x)-\Phi_{j}(x)}V_{k,h}^{i,j}(x)V_{p,q}^{i,j}(x)+\sum_{i\neq j}\frac{f( \Phi_{j}(x))}{\Phi_{j}(x)-\Phi_{i}(x)}V_{k,h}^{j,i}(x)V_{p,q}^{j,i}(x).
\end{split}
\end{equation}
From \eqref{eq:Vdef}, we can easily check that $V_{k,h}^{i,j}(x)=V_{k,h}^{j,i}(x)$ for all $1\leq i,j\leq n$ and $1\leq k\leq h\leq n$, which implies that identity \eqref{eq:Fpartialv01} can be rewritten as follows
\begin{align}\label{eq:Pikhpqdefprev}
\sum_{i=1}^{n}f(\Phi_{i}(x))\frac{\partial^{2}\Phi_{i}}{\partial x_{k,h}\partial x_{p,q}}(x)
  &=\sum_{j\neq i}\frac{f (\Phi_{i}(x))-f(\Phi_j(x))}{\Phi_{i}(x)-\Phi_{j}(x)}V_{k,h}^{i,j}(x)V_{p,q}^{i,j}(x).
\end{align}
Thus, by \eqref{eq:VandW}, the functional
\begin{align}\label{eq:Pikhpqdef}
\Pi_{k,h}^{p,q}[f](x)
  &:=\sum_{i=1}^{n}f(\Phi_{i}(x))\frac{\partial\Phi_{i}}{\partial x_{k,h}\partial x_{p,q}}(x),
\end{align}
satisfies
\begin{align*}
\sum_{k\leq h}\sum_{p\leq q}\big|\Pi_{k,h}^{p,q}[f](x)\big|^2
  &=\sum_{j_1\neq i_1}\sum_{j_2\neq i_2}\frac{f(\Phi_{i_1}(x))-f(\Phi_{j_1}(x))}{\Phi_{i_1}(x)-\Phi_{j_1}(x)}\frac{f( \Phi_{i_2}(x))-f(\Phi_{j_2}(x))}{\Phi_{i_2}(x)-\Phi_{j_2}(x)}\\
	&\hspace{6cm}\times \bigg(\delta_{i_1,i_2}\delta_{j_1,j_2}+\delta_{i_1,j_2}\delta_{j_1,i_2}\bigg)^2,
\end{align*}
which simplifies to 
\begin{align}
\sum_{k\leq h}\sum_{p\leq q}\big|\Pi_{k,h}^{p,q}[f](x)\big|^2
  &=2\sum_{i\neq j}\bigg(\frac{f( \Phi_{i}(x))-f(\Phi_{j}(x))}{\Phi_{i}(x)-\Phi_{j}(x)}\bigg)^2\label{eq:fphisum24}.
\end{align}

We end this section by proving the following result, which will be repeatedly used throughout the paper and holds for any  standard Gaussian  orthogonal ensamble.
\begin{Lemma}\label{lem:Wignernorm}
Let $A(n)$ be a standard Gaussian  orthogonal ensamble of dimension $n$. Then, for every $\gamma,\nu>1$, $M>0$  satisfying $\nu\leq \gamma$, and 
every continuously differentiable function $f:\R\rightarrow\R$ such that $f$ and $f^{\prime}$ have polynomial growth, there exists a constant $C>0$, such that 
\begin{align}\label{ineq:lemmawignerbound}
\sup_{n\geq 1}\sup_{z\in[0,M]}\frac{1}{n}\sum_{i=1}^{n}\Norm{f(\Phi_{i}(zA(n)))^{2}}_{{\rm L}^{\gamma}(\Omega)}^{\nu}
  &\leq C.
\end{align}
and
\begin{align}\label{ineq:lemmawignerbound2}
\sup_{n\geq 1}\sup_{z\in[0,M]}\frac{1}{n^2}\sum_{i\neq j}\bigg\|\bigg(\frac{f( \Phi_{i}(zA(n)))-f(\Phi_{j}(zA(n)))}{\Phi_{i}(zA(n))-\Phi_{j}(zA(n))}\bigg)^{2}\bigg\|_{{\rm L}^{\gamma}(\Omega)}^{\nu}
  &\leq C
\end{align}
\end{Lemma} 

\begin{proof}
First we prove \eqref{ineq:lemmawignerbound}. Since $f$ has polynomial growth, there exists $a\in\N$ and a constant $C_f>0$ that only depends on $f$, such that $|f(zx)|\leq C_f(1+|x|^{2a})$. In other words,  it is enough to show that there is $C_1>0$ such that
\begin{align}\label{ineq:lemmawignerboundaux}
\sup_{n\geq 1}\frac{1}{n}\sum_{i=1}^{n}\Norm{(\Phi_{i}(A(n)))^{2a}}_{{\rm L}^{\gamma}(\Omega)}^{\nu}
  &\leq C_1,
\end{align}
for all $a>1$. Notice that 
\begin{align*}
\Norm{(\Phi_{i}(A(n)))^{2a}}_{{\rm L}^{\gamma}(\Omega)}^{\nu}
  &=\E\big[(\Phi_{i}(A(n)))^{2a\gamma}\big]^{\frac{\nu}{\gamma}},
\end{align*}
which by Jensen's inequality, leads to 
\begin{align*}
\frac{1}{n}\sum_{i=1}^{n}\Norm{(\Phi_{i}(A(n)))^{2a}}_{{\rm L}^{\gamma}(\Omega)}^{\nu}
  &=\frac{1}{n}\sum_{i=1}^{n}\E\big[(\Phi_{i}(A(n)))^{2a\gamma}\big]^{\frac{\nu}{\gamma}}
	\leq \bigg(\frac{1}{n}\sum_{i=1}^{n}\E\big[(\Phi_{i}(A(n)))^{2a\gamma}\big]\bigg)^{\frac{\nu}{\gamma}}.
\end{align*}
From \cite[Lemma~2.1.6]{anderson2010introduction}, it follows that for all positive integer $\ell\in\N$, the sequence $\frac{1}{n}\E[{\rm Tr}(A(n)^{\ell})]$ converges to the moment of order $\ell$ of the semicircle distribution $\mu_1^{\text{sc}}$. As a consequence, the term in the right-hand side of the previous inequality converges to 
$$\bigg(\int_{[-2,2]}|x|^{2a\gamma}\mu_1^{sc}(\ud x)\bigg)^{\frac{\nu}{\gamma}},$$
which gives the desired result.

 In order to prove \eqref{ineq:lemmawignerbound2}, we use the identity 
$$\frac{f(x)-f(y)}{x-y}=\int_{0}^{1}f^{\prime}(\theta x+(1-\theta)y)\ud\theta,$$ 
to write
\begin{align*}
\bigg|\frac{f(\Phi_{i}(zA(n)))-f(\Phi_{j} (zA(n)))}{\Phi_{i}(zA(n))-\Phi_{j}(zA(n))}\bigg|
  &\leq \int_{0}^{1}\big|f^{\prime}\Big(\theta \Phi_{i}(zA(n))+(1-\theta)\Phi_{j}(zA(n))\Big)\big|\ud\theta.
\end{align*}
Since $f\in\Pc$, there exists a constant $K_{f}>0$ and $b\in\N$, such that $|f|\leq K_f(1+|x|^{b})$. Thus, 
\begin{align*}
\bigg|\frac{f(\Phi_{i}(zA(n)))-f(\Phi_{j} (zA(n)))}{\Phi_{i}(zA(n))-\Phi_{j}(zA(n))}\bigg|
  &\leq K_f+K_f\int_{0}^{1}\big|(\theta \Phi_{i}(zA(n)+(1-\theta)\Phi_{j}(zA(n))\big|^b\ud \theta.
\end{align*}
After applying the binomial theorem, integrating the variable $\theta$ and using the bound $|z|\leq T$, we deduce that there exist $K>0$ such that
\begin{align*}
\bigg|\frac{f(\Phi_{i}(zA(n)))-f(\Phi_{j} (zA(n)))}{\Phi_{i}(zA(n))-\Phi_{j}(zA(n))}\bigg|
  &\leq K\Big(1+\big|\Phi_{i}(A(n)\big|^b+\big|\Phi_{j}(A(n)\big|^b\Big),
\end{align*}
which implies that 
\begin{align*}
\bigg|\frac{f(\Phi_{i}(zA(n)))-f(\Phi_{j} (zA(n)))}{\Phi_{i}(zA(n))-\Phi_{j}(zA(n))}\bigg|^2
  &\leq C_2\Big(1+\big|\Phi_{i}(A(n)\big|^{2b}+\big|\Phi_{j}^n(A(n)\big|^{2b}\Big),
\end{align*}
for some constant $C_2>0$ that only depending on $T$ and $f$. The inequality  in \eqref{ineq:lemmawignerbound2} then follows from the inequality in \eqref{ineq:lemmawignerboundaux}. The proof is now complete.
\end{proof}

\section{Asymptotic behavior of the covariance of $X_F$}\label{sec:covariance}
In this section we prove Theorem \ref{theorem:covariance}. To achieve this, we will first establish some smoothness properties (in the Malliavin sense) for $(\Phi_{1}(Y^{(n)}(t)),\dots,\Phi_{n}(Y^{(n)}(t)))$.  Let us recall the definition of the matrix valued Gaussian process $Y^{(n)}$ in \eqref{eq:Y},  as well as the eigenvalue functions $\Phi_{1},\dots, \Phi_{n}$ and the set of good matrices $G$, defined in Section \ref{sec:eigenvalues}. Since we will constantly deal with random variables involving the derivatives of the functions $\Phi_1,\dots, \Phi_n$ (which are functions only defined in the open dense subset $G$ of $\R^{d(n)}$ with $d(n) = n(n+1)/2$), we will use the following notation: for every real function $h:G\rightarrow\R$, defined only in an open dense subset $G\subset\R^{d(n)}$, we have that $\Pb[Y^{(n)}(t)\in G]=1$, and consequently, the random variable $h(Y^{(n)}(t))$ is well defined $\Pb$-almost everywhere, provided that $R(t,t)>0$. This justifies the use of the notation 
\begin{align*}
h(A)
  :=\left\{\begin{array}{cc}h(A)&\ \ \text{ if }\ A\in G\\0&\ \ \ \ \ \ \text{ if } \ A\in \R\backslash G.\end{array}\right.
\end{align*}

\begin{Lemma}\label{thm:Malliavinderivative}
For every $1\leq i\leq n$, the random variable $\Phi_{i}(Y^{(n)}(t))$ is twice Malliavin differentiable. The first and second Malliavin derivatives of $\Phi_{i}(Y^{(n)}(t))$, are given by $D\Phi_{i}(Y^{(n)}(t))=\{u_{k,h}(t) ; k\leq h\}$ and $D^2\Phi_{i}(Y^{(n)}(t))=\{u_{k,h}^{p,q}(t) ; k\leq h,\ p\leq q\}$, where 
\[
u_{k,h}(t):=\frac{\partial\Phi_{i}}{\partial x_{k,h}}(Y^{(n)}(t))\Indi{[0,t]}\quad\textrm{ and }\quad
u_{k,h}^{p,q}(t) 
  :=\frac{\partial^{2}\Phi_{i}}{\partial x_{k,h}\partial x_{p,q}}(Y^{(n)}(t))\Indi{[0,t]}^{\otimes 2}.
\]
\end{Lemma}
\begin{proof}
Let $A=\{A_{k,h}; k\leq h\}\in {\rm L}^{2}(\Hg^{d(n)})$ and $B=\{B_{k,h}^{p,q}; k\leq h, p\le q\}\in {\rm L}^{2}(\Hg^{d(n)})$ be defined as $A_{k,h}:=u_{k,h}(t)$ and $B_{k,h}^{p,q}:=u_{k,h}^{p,q}(t)$. Let $p_{\varepsilon}$ denote the $d(n)$-dimensional Gaussian kernel of variance $\epsilon$, defined by $p_{\varepsilon}(x):=(2\pi\varepsilon)^{-\frac{d(n)}{2}}\exp\{-\frac{\Abs{x}^2}{2\varepsilon}\}$. Then, the random variable $\Phi_{i}*p_{\varepsilon}(Y^{(n)}(t))$ is infinitely Malliavin differentiable and satisfies 
\[
\Phi_{i}*p_{\varepsilon}(Y^{(n)}(t))\xrightarrow[]{{\rm L}^{2}(\Omega)}\Phi_{i}(Y^{(n)}(t)).
\]
Thus, in order to prove the statement, it is enough  to show that 
\begin{equation}
D\Phi_{i}*p_{\varepsilon}(Y^{(n)}(t))\xrightarrow[]{{\rm L}^{2}(\Omega;\Hg^{d(n)})}A \quad \textrm{ and }\quad
D^2\Phi_{i}*p_{\varepsilon}(Y^{(n)}(t))\xrightarrow[]{{\rm L}^{2}(\Omega;\Hg^{d(n)})^{\otimes 2})}
 B.\label{eq:DPHImollifiedprime}
\end{equation}
In order to do so, we observe that  $D\Phi_{i}*p_{\varepsilon}(Y^{(n)}(t))=\{v_{k,h}(\varepsilon;t)\ ;\ k\leq h\}$ and $D^2\Phi_{i}*p_{\varepsilon}(Y^{(n)}(t))=\{v_{k,h}^{p,q}(\varepsilon;t)\ ;\ k\leq h,\ p\leq q\}$, where 
\[
v_{k,h}(\varepsilon;t):=\frac{\partial(\Phi_{i}*p_{\varepsilon})}{\partial x_{k,h}}(Y^{(n)}(t))\Indi{[0,t]},\quad \textrm{ and }\quad
v_{k,h}^{p,q}(\varepsilon;t):=\frac{\partial^{2}(\Phi_{i}*p_{\varepsilon})}{\partial x_{k,h}\partial x_{p,q}}(Y^{(n)}(t))\Indi{[0,t]}^{\otimes 2}.
\]
Hence, provided that we deduce
\begin{equation}\label{eq:derivconvPhi}
v_{k,h}(\varepsilon;t)
  =\frac{\partial\Phi_{i}}{\partial x_{k,h}}*p_{\varepsilon}(Y^{(n)}(t))\Indi{[0,t]},\quad \textrm{ and }\quad
v_{k,h}^{p,q}(\varepsilon;t) 
  =\frac{\partial^{2}\Phi_{i}}{\partial x_{k,h}\partial x_{p,q}}*p_{\varepsilon}(Y^{(n)}(t))\Indi{[0,t]}^{\otimes 2},
\end{equation}
we obtain   \eqref{eq:DPHImollifiedprime} by using the well-known fact 
$$\|p_{\varepsilon}*f-f\|_{{\rm L}^{2}(\R^{d(n)},\mu)}\rightarrow0, $$
as $\varepsilon$ goes to 0, 
for every measure $\mu$ defined in $\R^{d(n)}$ and every $f\in {\rm L}^{2}(\R^{d(n)},\mu)$. Notice that \eqref{eq:derivconvPhi} is equivalent to 
\begin{align}\label{eq:DPhimollified}
\frac{\partial(\Phi_{i}*p_{\varepsilon})}{\partial x_{k,h}}=\frac{\partial\Phi_{i}}{\partial x_{k,h}}*p_{\varepsilon},\ \ \ \ \ \ \ \ \ \text{and}\ \ \ \ \ \ \ \ \ \frac{\partial^{2}(\Phi_{i}*p_{\varepsilon})}{\partial x_{k,h}\partial x_{p,q}} 
  &=\frac{\partial^{2}\Phi_{i}}{\partial x_{k,h}\partial x_{p,q}}*p_{\varepsilon}.
\end{align}
In order to show \eqref{eq:DPhimollified}, we proceed as follows. Denote by $e^{p,q}=\{e_{k,h}^{p,q}; 1\leq k\leq h\leq h\}$ the $(k,h)$-canonical element of $\R^{d(n)}$, given by $e_{k,h}^{p,q}:=\delta_{k,p}\delta_{h,q}$. For every $y\in\R^{d(n)-1}$ of the form $y=\{y_{k,h};  1\leq k\leq h\leq n\ \text{ and }\ (k,h)\neq(p,q)\}$, consider the linear mapping $\pi^{p,q,y}:\R\rightarrow\R^{d(n)}$, given by $\pi^{p,q,y}(z)=\{\pi_{k,h}^{p,q,y}(z); k\leq h\}$, with
\begin{align*}
\pi_{k,h}^{p,q,y}(z)
  &:=\left\{\begin{array}{cc}y_{k,h}  &\text{ if }\ \ (k,h)\neq(p,q),\\ z&\text{ if }\ \ (k,h)=(p,q).\end{array}\right.
\end{align*}
Notice that for all $1\leq i\leq n$, the function $\Phi_i$ is infinitely differentiable in the complement of the set $\mathcal{S}_{\text{deg}}$ of $n\times n$ symmetric matrices with at least one repeated eigenvalue. In \cite[Proposition~4.5.]{jaramillo2018collision}, it was shown that the set $\mathcal{S}_{\text{deg}}^c$ is contained in the image of a smooth function defined over $\R^{d(n)-2}$. From this observation it easily follows that for almost all $y\in\R^{d(n)-1}$, the function $\Phi_{i}\circ \pi_{k,h}^{p,q,y}$ is infinitely differentiable. As a consequence, for every $x\in\R^{d(n)}$,
\begin{align*}
\frac{\partial(\Phi_{i}*p_{\varepsilon})}{\partial x_{k,h}}(x)
  &=\int_{\R^{d(n)}}\Phi_{i}(x-\xi)\frac{\partial p_{\varepsilon}}{\partial x_{k,h}}(\xi)\ud \xi\\
	&=\int_{\R^{d(n)-1}}\int_{\R}\Phi_{i}(x-\pi^{k,h,y}(z))\frac{\ud p_{\varepsilon}}{\ud z}(\pi^{k,h,y}(z))\ud z\ud y\\
	&=\int_{\R^{d(n)-1}}\int_{\R}\frac{\partial \Phi_{i}}{\partial x_{k,h}}(x-\pi^{k,h,y}(z))p_{\varepsilon}(\pi^{k,h,y}(z))\ud z\ud y.
\end{align*}
where the integration by parts in the last equality, is justified by the fact that the mapping $z\mapsto \Phi_{i}\circ \pi_{k,h}^{p,q,y}(z)$ is infinitely differentiable for almost all $y\in\R^{d(n)}$. From here, it easily follows that $\frac{\partial(\Phi_{i}*p_{\varepsilon})}{\partial x_{k,h}}(x)=\frac{\partial\Phi_{i}}{\partial x_{k,h}}*p_{\varepsilon}(x)$. To prove the second inequality in \eqref{eq:DPhimollified}, we proceed similarly, but replacing the function $\Phi_{i}$, with $\frac{\partial\Phi_{i}}{\partial x_{k,h}}$.
\end{proof}

Before proving Theorem~\ref{theorem:covariance}, we establish the following auxiliary lemma.

\begin{Lemma}
\label{lem:tauvoic}
Assume that $\xi$ and $\tilde{\xi}$ are free standard semicircular non-commutative random variables. If $\varphi,\psi\in \mathcal{C}(\mathbb{R};\mathbb{R})$ and $z\in[0,1)$, then
\begin{equation}
\label{eq:tauvoic}
    \tau\bigg[\varphi\left(z\xi+\sqrt{1-z^2}\widetilde{\xi}\right)\psi(\xi)\bigg]
  =\int_{[-2,2]^2}\varphi(x)\psi(y)K_z(x,y)\mu_1^{sc}(\ud x)\mu_1^{sc}(\ud y),
\end{equation}
with $K_z(x,y)$ defined as in \eqref{eq:Kdef0}.
\end{Lemma}

\begin{proof}
For ease of notation, let $a = z\xi+\sqrt{1-z^2}\widetilde{\xi}$ and $b = \tilde{\xi}$. Since both $a$ and $b$ are (correlated) standard semicircular non-commutative random variables, a straightforward application of functional calculus implies that $\varphi(a) = (\varphi\circ\Indi{[-2,2]})(a)$ and $\psi(b) = (\psi\circ\Indi{[-2,2]})(b)$. Similarly, observe that the right hand side of \eqref{eq:tauvoic} remains the same if we replace $\varphi$ with $\varphi\circ\Indi{[-2,2]}$ and $\psi$ with $\psi\circ\Indi{[-2,2]}$. Hence, without of generality, we can assume that both $\varphi$ and $\psi$ are supported over $[-2,2]$.

Let $U_m(x)$ denote the $m$-th Chebyshev polynomial of the second kind on $[-2,2]$, defined by \eqref{eq:Chebyshevdef}. By the Stone-Weierstrass theorem, we can assume without loss of generality that $\varphi(x)=U_{m_{1}}(x)$ and $\psi(y)=U_{m_2}(y)$ for some $m_1,m_2\in\N$. Since the Chebyshev polynomials form an orthonormal system with respect to $\mu_1^{sc}(\ud x)$, 
the measure $\kappa_z(\ud x,\ud y)$ defined by 
\begin{align*}
\kappa_z(\ud x,\ud y)
  &:=\Indi{[-2,2]^2}(x,y)K_{z}(x,y)\mu_1^{sc}(\ud x)\mu_1^{sc}(\ud y), 
\end{align*}
satisfies 
\begin{align}\label{eq:UorthtauRHS}
\int_{\R^2}\varphi(x)\psi(y)\kappa_z(\ud x,\ud y)
  &=\int_{\R^2}U_{m_1}(x)U_{m_2}(y)\kappa_z(\ud x,\ud y)=\delta_{m_1,m_2}z^{m_{1}}.
\end{align}
On the other hand, by relation \eqref{eq:isometrysemicircular}, we have 
\begin{align}\label{eq:Uorthtau}
\tau\left[U_{m_1}\Big(z\xi+\sqrt{1-z^2}\widetilde{\xi}\Big)U_{m_{2}}(\xi)\right]
  &=\delta_{m_1,m_2}\tau\left[\Big(z\xi+\sqrt{1-z^2}\widetilde{\xi}\Big)^*\xi\right]^{m_1}=\delta_{m_1,m_2}z^{m_{1}}.
\end{align}
By combining the identities \eqref{eq:UorthtauRHS} and \eqref{eq:Uorthtau}, we get 
\begin{align*}
\tau\left[\varphi\Big(z\xi+\sqrt{1-z^2}\widetilde{\xi}\Big)\psi(\xi)\right]
  &=\int_{\R^2}\varphi(x)\psi(y)\kappa_z(\ud x,\ud y),
\end{align*}
as required.
\end{proof}

\noindent We are now in position of proving Theorem \ref{theorem:covariance}.
\begin{proof}[Proof of Theorem \ref{theorem:covariance}]
\noindent  We recall the smoothness properties (in the Malliavin sense) of the random vector $$\Phi_{1}(Y^{(n)}(t)),\dots, \Phi_{n}(Y^{(n)}(t)),$$
which were introduced in Lemma \ref{thm:Malliavinderivative} and will be repeatedly used in the sequel. By \eqref{eq:deltaDFI}, every centered random variables $F,G\in \D^{1,2}$, satisfy
\begin{align*}
\E[FG]
  &=\E[-\delta (D \mathcal{L}^{-1}F)G]	=\E\left[\Ip{-D  \mathcal{L}^{-1}F,DG}_{\Hg^{d(n)}}\right].
\end{align*}
In particular, for every $f,g\in\Pc$ and $s,t>0$, 
\begin{align}\label{eq:covfirstexp}
\E\left[Z_{f}^{(n)}(t)Z_{g}^{(n)}(s)\right]
  &=\E\left[\Ip{-D  \mathcal{L}^{-1}Z_{f}^{(n)}(t),DZ_{g}^{(n)}(s)}_{\Hg^{d(n)}}\right].
\end{align}
By \eqref{eq:Mehler} and \eqref{eq:Mehler2}, we get
\begin{align*}
-D  \mathcal{L}^{-1}Z_{f}^{(n)}(t)
  &=\int_{0}^{\infty}DP_{\theta}\big[Z_{f}^{(n)}(t)\big]\ud \theta=\int_{0}^{\infty}e^{-\theta}P_{\theta}\big[DZ_{f}^{(n)}(t)\big]\ud \theta.
\end{align*}
On the other hand, Lemma \ref{thm:Malliavinderivative} implies  that $DZ_{f}^{(n)}(t)=\{v_{k,h}(t); 1\leq k,h\leq n\}$, with 
\begin{align*}
v_{k,h}(t)
  &=\frac{1}{\sqrt{n}}\sum_{i=1}^{n}f^{\prime}\big(\Phi_{i}(Y^{(n)}(t))\big)\frac{\partial\Phi_{i}}{\partial x_{k,h}}(Y^{(n)}(t))\Indi{[0,t]}.
\end{align*}
Using equation \eqref{eq:D1PhiUpc} and denoting by $U^{*}(Y^{(n)}(t))$ the transpose of $U(Y^{(n)}(t))$, we can rewrite $v_{k,h}(t)$ as
\[
\begin{split}
v_{k,h}(t)
	&=\frac{1}{\sqrt{n}}\sum_{i=1}^{n}f^{\prime}\big(\Phi_{i}(Y^{(n)}(t))\big)\bigg(\big(U_{h,i}U_{i,k}^*+U_{k,i}U_{i,h}^*\big)(Y^{(n)}(t))\Indi{\{k\neq h\}}\\
	&\hspace{7cm}+\sqrt{2}U_{k,i}U_{i,k}^*(Y^{(n)}(t))\Indi{\{k=h\}}]\bigg)\Indi{[0,t]}\\
	&=\frac{1}{\sqrt{n}}\left(\Big(f^{\prime}(Y^{(n)}(t))_{h,k}+f^{\prime}(Y^{(n)}(t))_{k,h}\Big)\Indi{\{k\neq h\}}+\sqrt{2}f^{\prime}(Y^{(n)}(t))_{k,k}\Indi{\{k=h\}}\right)\Indi{[0,t]}.
\end{split}
\]
Therefore, using Mehler's formula \eqref{eq:Mehler} as well as the fact that $f^{\prime}(Y^{(n)}(t))$ is self-adjoint, we deduce that $-D \mathcal{L}^{-1}Z_{f}^{(n)}(t)=\{u_{k,h}(t); 1\leq k,h\leq n\}$, where
\begin{align*}
u_{k,h}(t)
	&:=\frac{\eta_{k,h}}{\sqrt{n}}\int_{0}^{\infty}e^{-\theta}\widetilde{\E}\Big[\Big(f^{\prime}\big(e^{-\theta}Y^{(n)}(t)+\sqrt{1-e^{-2\theta}}\widetilde{Y}^{(n)}(t)\big)\Big)_{k,h}\Big]\ud\theta\Indi{[0,t]},
\end{align*}
where $\widetilde{Y}^{(n)}$ is an independent copy of $Y^{(n)}$, $\eta_{k,h}:=2\Indi{\{k\neq h\}}+\sqrt{2}\Indi{\{k=h\}}$ and $\widetilde{\E}$ denotes the expectation with respect to $\widetilde{Y}^{(n)}$. Similarly, we deduce  $DZ_{g}^{(n)}(s)=\{\omega_{k,h}(s); 1\leq k,h\leq n\}$, where
\begin{align*}
\omega_{k,h}(s)
	&:=\frac{\eta_{k,h}}{\sqrt{n}}(g^{\prime}(Y^{(n)}(s)))_{k,h}\Indi{[0,s]}.
\end{align*}
As a consequence, we have 
\begin{multline*}
\E\left[\Ip{-D  \mathcal{L}^{-1}Z_{f}^{(n)}(t),DZ_{g}^{(n)}(s)}_{\Hg^{\otimes d(n)}}\right]\\
\begin{aligned}
  &=\int_{0}^{\infty}e^{-\theta}\E\left[\widetilde{\E}\left[\sum_{k\leq h}\frac{\eta_{k,h}^2}{n}\left(f^{\prime}(e^{-\theta}Y^{(n)}(t)+\sqrt{1-e^{-2\theta}}\widetilde{Y}^{(n)}(t))\right)_{k,h}\left(g^{\prime}(Y^{(n)}(s))\right)_{k,h}\right]\right]\ud\theta\\
	&=\frac{2}{n}\int_{0}^{\infty}e^{-\theta}\E\left[\text{Tr}\Big(f^{\prime}(e^{-\theta}Y^{(n)}(t)+\sqrt{1-e^{-2\theta}}\widetilde{Y}^{(n)}(t))g^{\prime}(Y^{(n)}(s))\Big)\right]\ud\theta.
\end{aligned}
\end{multline*}
Hence, making the change of variable $z:=e^{-\theta}$, we get 
\begin{equation}
\label{eq:varDLinvcov}
\begin{split}
\E&\left[\Ip{-D  \mathcal{L}^{-1}Z_{f}^{(n)}(t),DZ_{g}^{(n)}(s)}_{\Hg^{\otimes d(n)}}\right]\\
  &\hspace{3cm}=\frac{2}{n}\int_{0}^{1}\E\left[\text{Tr}\Big(f^{\prime}\Big(zY^{(n)}(t)+\sqrt{1-z^2}\widetilde{Y}^{(n)}(t))g^{\prime}(Y^{(n)}(s)\Big)\Big)\right]\ud z.
  \end{split}
\end{equation}
Let ${A}(n)$ and $\widetilde{A}(n)$ be two independent standard Gaussian orthogonal ensembles and recall the definitions of  $\sigma_s$ and $\rho_{s,t}$  in \eqref{eq:Rfunctionals}. It is not difficult to deduce
\begin{align*}
\bigg(zY^{(n)}(t)+\sqrt{1-z^2}\widetilde{Y}^{(n)}(t), Y^{(n)}(s)\bigg)
  &\stackrel{(d)}{=}\bigg(\sigma_t\left(\rho_{s,t}zA(n)+\sqrt{1-\rho_{s,t}^2z^2}\widetilde{A}(n)\right),\sigma_sA(n)\bigg),
\end{align*}
where $``\stackrel{(d)}{=}"$ means identity in distribution. Thus, by Voiculescu theorem (see for instance \cite[Theorem~3.3]
{anderson2010introduction}), we get 
\begin{multline}\label{eq:tracelimit}
\lim_{n\rightarrow\infty}\frac{1}{n}\E\left[\text{Tr}\left(f^{\prime}\Big(zY^{(n)}(t)+\sqrt{1-z^2}\widetilde{Y}^{(n)}(t)\Big)\right)g^{\prime}\Big(Y^{(n)}(s)\Big)\right]\\
\begin{aligned}
  &=\lim_{n\rightarrow\infty}\frac{1}{n}\E\left[\text{Tr}\left(f^{\prime}\left(\sigma_t\left(\rho_{s,t}zA(n)+\sqrt{1-\rho_{s,t}^2z^2}\widetilde{A}(n)\right)\right)g^{\prime}\big(\sigma_sA(n)\big)\right)\right]\\
  &=\tau\left[(f^{\prime}\circ m_{\sigma_t})\left((z\rho_{s,t})\xi+\sqrt{1-(z\rho_{s,t})^2}\widetilde{\xi}\right)(g^{\prime}\circ m_{\sigma_s})(\xi)\right],
\end{aligned}
\end{multline}
where $m_{\sigma}$ denotes the multiplication function $m_{\sigma}(y):=\sigma y$,  and  $\xi, \widetilde{\xi}$ are self-adjoint free random variables with standard semicircular distribution, defined on a non-commutative probability space $(\mathcal{A},\tau)$. By Lemma~\ref{lem:tauvoic}
\begin{align*}
\tau\bigg[\varphi\left(z\xi+\sqrt{1-z^2}\widetilde{\xi}\right)\psi(\xi)\bigg]
  &=\int_{[-2,2]^2}\varphi(x)\psi(y)K_z(x,y)\mu_1^{sc}(\ud x)\mu_1^{sc}(\ud y).
\end{align*}
In addition, by the Cauchy-Schwarz inequality and Wigner's theorem,
\begin{multline*}
\frac{1}{n}\E\left[\text{Tr}\left(f^{\prime}\left(\sigma_t\left(\rho_{s,t}zA(n)+\sqrt{1-\rho_{s,t}^2z^2}\widetilde{A}(n)\right)\right)g^{\prime}\big(\sigma_sA(n)\big)\right)\right]\\
\begin{aligned}
  &\leq \bigg(
  \frac{1}{n}\E\big[\text{Tr}\big(f^{\prime}\big(\sigma_tA(n)\big)^2\big)\big]
  \bigg)^{\frac{1}{2}}\bigg(
  \frac{1}{n}\E\big[\text{Tr}\big(g^{\prime}\big(\sigma_sA(n)\big)^2\big)\big]
  \bigg)^{\frac{1}{2}}
    \leq C_{s,t},
\end{aligned}
\end{multline*}
for some constant $C_{s,t}>0$ independent of $n$. Therefore, using the dominated convergence theorem, as well as \eqref{eq:covfirstexp}, \eqref{eq:varDLinvcov} and \eqref{eq:tracelimit}, we deduce that
\begin{align*}
\lim_{n\rightarrow\infty}\E\Big[Z_{f}^{(n)}(t)Z_{g}^{(n)}(s)\Big]
  &=\lim_{n\rightarrow\infty}\E\Big[\Ip{-D \mathcal{L}^{-1}Z_{f}^{(n)}(t),DZ_{g}^{(n)}(s)}_{\Hg^{\otimes d(n)}}\Big]\\
	&=2\int_{0}^{1}\int_{[-2,2]^{2}}f^{\prime}(\sigma_sx)g^{\prime}(\sigma_ty)K_{z\rho_{s,t}}(x,y)\mu_1^{sc}(\ud x)\mu_1^{sc}(\ud y)\ud z.
\end{align*}
Making the changes of variable $\tilde{x}:=\sigma_s x$ and  $\tilde{y}:=\sigma_t y$, we obtain 
\begin{align*}
\lim_{n\rightarrow\infty}\E\Big[Z_{f}^{(n)}(t)Z_{g}^{(n)}(s)\Big]
  &=2\int_{0}^{1}\int_{\R^{2}}f^{\prime}(\tilde{x})g^{\prime}(\tilde{y})K_{z\rho_{s,t}}(\tilde{x}/\sigma_s,\tilde{y}/\sigma_t)\mu_{\sigma_s}^{sc}(\ud\tilde{x})\mu_{\sigma_t}^{sc}(\ud\tilde{y})\ud z.
\end{align*}
Theorem~\ref{theorem:covariance} easily follows from the previous expression. The proof is now complete.
\end{proof}

\section{Convergence of finite dimensional distributions}\label{sec:fdd}
In this section we prove the stable convergence of the finite dimensional distributions of $Z^{(n)}_{F}$, to those of $\Lambda_{F}$, for $F\in \mathcal{P}^r$ with $r\ge 1$, 
and find bounds for the distance in total variation of $Z^{n}_{f}(t)$ to its limit distribution, with $f\in \mathcal{P}$.

\begin{Proposition}\label{Prop:main}
Assume that  the process $X$ satisfies the conditions \textbf{(H1)} and \textbf{(H2)}, introduced in Section \ref{Sec:contrib}. We recall as well the definition of $\mathcal{P}$, given by \eqref{eq:Pmathcaldef}. For every $r,y\in\N$ and $F=(f_1,\dots, f_{r})\in \Pc^r$, and $t_1,\dots, t_{\ell}\geq0$,  there exists $C>0$, such that
\begin{align*}
(Z_{F}^{{(n)}}(t_1),\dots, Z_{F}^{(n)}(t_\ell))\stackrel{\mathcal{S}}{\longrightarrow}(\Lambda_{F}(t_1),\dots, \Lambda_{F}(t_\ell)).
\end{align*}
Moreover,  for $f\in \mathcal{P}$, we have
\begin{align*}
d_{TV}\left(\mu_{Z_f^{(n)}}(t),\mu_{\Lambda_{f}(t)}\right)
  &\leq \frac{C}{\sqrt{n}},
\end{align*}
for some constant $C>0$ independent of $n$.

\end{Proposition}
\begin{proof} Let $T>0$ be fixed and  denote by $\Cc[0,T]$  the set of continuous functions in $[0,T]$. Let us consider a function $g:(\R^{r})^\ell\rightarrow\R$, as well as a bounded $\Fc$-measurable random variable $M$. We first show that 
\begin{align}\label{eq:stablegvectF}
\lim_{n\rightarrow\infty}\E\left[g(Z_{F}^{(n)}(t_1),\dots, Z_{F}^{(n)}(t_\ell))M\right]
  =\E\left[g(\Lambda_{F}^{n}(t_1),\dots, \Lambda_{F}(t_\ell))\right]\E\left[M\right],
\end{align}
for all $t_1,\dots, t_{\ell}\geq0$, $j\in\N$. Since $M$  is $\Fc$-measurable and bounded, there exists a sequence of natural numbers $\{l_{m}\ ;\ m \geq 1\}$, as well as a collection of continuous and bounded functions $h_{m}:\R^{l_m d(l_m)}\rightarrow\R$ and random variables of the form 
$$M_m=h_m\Big(X_{i,j}(s_{1}^{m}),\dots,X_{i,j}(s_{l_m}^{m});1\leq i\leq j\leq l_{m}\Big),$$
with $s_{1}^m,\dots,s_{l_m}^m>0$, such that $M_{m}\stackrel{{\rm L}^{2}(\Omega)}{\longrightarrow}M$ as $m\rightarrow\infty$. Hence, by applying an approximation argument, we deduce that it suffices to show relation \eqref{eq:stablegvectF} for $M$ of the form $M=h(\eta),$ where $h:\R^{ld(L)}\rightarrow\R$, and 
$$\eta=\Big(X_{i,j}(s_{1}),\dots, X_{i,j}(s_{l}); 1\leq i\leq j\leq l \Big),$$
with $l\in\N$ and $s_{1},\dots,s_{l}>0$. Since $\Lambda_{F}$ is independent of $\Fc$, the problem is then reduced to show that
\begin{align}\label{eq:stableconv}
(Z_{F}^{(n)}(t_{1}),\dots, Z_{F}^{(n)}(t_{\ell}),\eta)
  &\stackrel{(d)}{\longrightarrow}(\Lambda_{F}(t_1),\dots, \Lambda_{F}(t_{\ell}),\eta).
\end{align}
To show the convergence \eqref{eq:stableconv}, it suffices to verify the conditions of Theorem \ref{thm:CLTWienerchaos}.

 Condition (i) follows directly from Theorem \ref{theorem:covariance}. In order to prove condition (ii), notice that by Lemma \ref{thm:Malliavinderivative}, for every $t\geq 0$ and $f\in\Pc$ we have that 
 \begin{equation}\label{DZ_f}
 D Z_{f}^{(n)}(t)=\Big\{n^{-\frac{1}{2}}\Psi_{k,h}[f^{\prime}](Y^{(n)}(t))\Indi{[0,t]}; 1\leq k,h\leq n\Big\},
 \end{equation} where 
$\Psi_{k,h}[f^{\prime}]$ is defined by \eqref{eq:psik1def}. 
Hence, by \eqref{eq:isimetrypsik1}, 
\begin{align*}
\Norm{D Z_{f}^{(n)}(t)}_{\Hg^{\otimes 	d(n)}}^2
  &=\frac{2\sigma_t^2}{n}\sum_{i=1}^{n}f^{\prime}(\Phi_{i}(Y^{(n)}(t)))^2.
\end{align*}
Therefore, using Lemma \ref{lem:Wignernorm}, we deduce that there exists a constant $C>0$, independent of $n$, such that 
\begin{align*}
\Norm{D Z_{f}^{(n)}(t)}_{{\rm L}^{4}(\Omega;\Hg^{\otimes 	d(n)})}^4
  &=\Norm{\Norm{D Z_{f}^{(n)}(t)}_{\Hg^{\otimes 	d(n)}}^2}_{{\rm L}^{2}(\Omega)}^2
  \leq \bigg\|\frac{2\sigma_t^2}{n}\sum_{i=1}^{n}f^{\prime}(\Phi_{i}(Y^{(n)}(t)))^2\bigg\|_{{\rm L}^{2}(\Omega)}^2\leq C,
\end{align*}
which implies condition (ii).

In order to deduce condition (iii), we first use Lemma \ref{thm:Malliavinderivative} to write 
$$D^2Z_{f}^{(n)}(t)=\left\{\frac{1}{n}\Big(\Psi_{k,h}^{p,q}[f^{\prime\prime}](Y^{(n)}(t))+\Pi_{k,h}^{p,q}[f^{\prime}](Y^{(n)}(t
))\Big)\Indi{[0,t]}^{\otimes 2};\ 1\leq k,h,p,q\leq n\right\},$$ 

where $\Psi_{k,h}^{p,q}[f^{\prime\prime}]$ and $\Pi_{k,h}^{p,q}[f^{\prime}]$ are given by \eqref{eq:psik2def} and \eqref{eq:Pikhpqdef}.
Therefore, by \eqref{eq:isimetrypsik2} and \eqref{eq:fphisum24}
\begin{multline}\label{ineq:D2XD2xeq}
\left\|D^2Z_{f}^{(n)}(t)\otimes_1 D^2Z_{f}^{(n)}(t)\right\|_{(\Hg^{d(n)})^{\otimes 2}}^2\\
  \begin{aligned}
	&=\frac{\sigma_{t}^4}{n^4}\sum_{\substack{k_1\leq h_1\\k_2\leq h_2}}\sum_{p\leq q}\big(\Psi_{k_1,h_1}^{p,q}[f^{\prime\prime}]\Psi_{k_2,h_2}^{p,q}[f^{\prime\prime}]+\Pi_{k_1,h_1}^{p,q}[f^{\prime}]\Psi_{k_2,h_2}^{p,q}[f^{\prime\prime}]\big)(Y^{(n)}(t))\\
	&+\frac{\sigma_{t}^4}{n^4}\sum_{\substack{k_1\leq h_1\\k_2\leq h_2}}\sum_{p\leq q}\big(\Psi_{k_1,h_1}^{p,q}[f^{\prime\prime}]\Pi_{k_2,h_2}^{p,q}[f^{\prime}]+\Pi_{k_1,h_1}^{p,q}[f^{\prime}]\Pi_{k_2,h_2}^{p,q}[f^{\prime}]\big)(Y^{(n)}(t)).
\end{aligned}
\end{multline}
By applying the Cauchy-Schwarz inequality in \eqref{ineq:D2XD2xeq}, it is straightforward to see
\begin{multline*}
\left\|D^2Z_{f}^{(n)}(t)\otimes_1 D^2Z_{f}^{(n)}(t)\right\|_{(\Hg^{d(n)})^{\otimes 2}}^2\\
  \begin{aligned}
	&\leq\frac{\sigma_{t}^4}{n^4}\sum_{\substack{k_1\leq h_1\\k_2\leq h_2}}\sum_{p\leq q}\big(\Psi_{k_1,h_1}^{p,q}[f^{\prime\prime}]^2+\Psi_{k_2,h_2}^{p,q}[f^{\prime\prime}]^2+\Pi_{k_1,h_1}^{p,q}[f^{\prime}]^2+\Pi_{k_2,h_2}^{p,q}[f^{\prime}]^2\big)(Y^{(n)}(t)),
\end{aligned}
\end{multline*}
which in turn implies that
\begin{align*}
\left\|D^2Z_{f}^{(n)}(t)\otimes_1 D^2Z_{f}^{(n)}(t)\right\|_{(\Hg^{d(n)})^{\otimes 2}}^2
   &\leq \frac{2\sigma_t^4}{n^4}\sum_{k\leq h}\sum_{p\leq q}\big(\Psi_{k,h}^{p,q}[f^{\prime\prime}]^2+\Pi_{k,h}^{p,q}[f^{\prime}]^2\big)(Y^{(n)}(t)).
\end{align*}
Relations \eqref{eq:isimetrypsik2} and \eqref{eq:fphisum24} allow us to write the previous inequality as
\[
\begin{split}
\left\|D^2Z_{f}^{(n)}(t)\otimes_1 D^2Z_{f}^{(n)}(t)\right\|_{(\Hg^{d(n)})^{\otimes 2}}^2
  &\leq \frac{8\sigma_t^4}{n^4}\sum_{i=1}^{n}f^{\prime\prime}(\Phi_{i}(Y^{(n)}(t)))^2\\
  &+\frac{4\sigma_t^4}{n^4}\bigg(\sum_{j\neq i}\frac{f^{\prime}(\Phi_{i}(Y^{(n)}(t)))-f^{\prime}(\Phi_{j}(Y^{(n)}(t)))}{\Phi_{i}(Y^{(n)}(t))-\Phi_{j}(Y^{(n)}(t))}\bigg)^2.
\end{split}
\]
Using the previous inequality, we deduce that if $\{A(n); n\geq 1\}$ is a standard Gaussian orthogonal ensemble, then
\begin{align*}
\left\|D^2Z_{f}^{(n)}(t)\otimes_1 D^2Z_{f}^{(n)}(t)\right\|_{(\Hg^{d(n)})^{\otimes 2}}^2
  &\leq \frac{8\sigma_t^4}{n^4}\sum_{i=1}^{n}\E\Big[(f^{\prime\prime}(\Phi_{i}(\sigma_tA(n)))^2\Big]\\
	&+\frac{4\sigma_t^4}{n^4}\sum_{j\neq i}\E\left[\left(\frac{f^{\prime}(\Phi_{i}(\sigma_tA(n)))-f^{\prime}(\Phi_{j}(\sigma_tA(n)))} {\Phi_{i}(\sigma_t A(n))-\Phi_{j}(\sigma_t A(n))}\right)^2\right].
\end{align*}
 Thus, by Lemma \ref{lem:Wignernorm} we get that 
\begin{align*}
\left\|D^2Z_{f}^{(n)}(t)\otimes_1 D^2Z_{f}^{(n)}(t)\right\|_{(\Hg^{d(n)})^{\otimes 2}}^2
  &\leq \frac{C}{n^2},
\end{align*}
for some constant $C>0$ independent of $n$. Thus Proposition \ref{Prop:main} follows directly from Theorem \ref{thm:CLTWienerchaos}.
\end{proof}
\section{Tightness property for $Z_{f}^{(n)}(t)$}\label{sec:tight}
Recall that the family of test functions $\mathcal{P}$ consists of functions with derivatives of order fourth with polynomial growth, see \eqref{eq:Pmathcaldef}. The goal of this section is to prove the following result
\begin{Lemma}
If $f\in \Pc$, then the process $\{Z_{f}^{(n)} ;\ n\geq1\}$, with  $Z_{f}^{(n)}:=(Z_{f}^{(n)}(t) ; t\geq0)$, is tight.
\end{Lemma}
\begin{proof}
Before delving further into the details of the proof, we would like to briefly mention the overall ideas we present. The foundations of the argument rely on the so called Billingsley criterion, in which the sequential compactness is proved by controlling the size of moments of the increments  of $Z_{f}^{(n)}(t)$. This will be achieved by decomposing  $Z_{f}^{(n)}(t)$ into two pieces: a generalized Skorohod integral on the variable $t$ and a Lebesgue integral, taken with respect to $t$ as well. The size of the moments of the Skorohod integral is estimated by integration by parts arguments, while the Lebesgue integral uses Malliavin calculus techniques.\\

In \cite[Lemma~3.1]{jaramillo2018convergence}, it was proved that the random variable $\int f(x)\mu_{t}^{(n)}(\ud x)$ satisfies the following stochastic equation 
\begin{equation}\label{eq:evolution}
\begin{split}
&\int f(x)\mu_{t}^{(n)}(\ud x)
  =f(0)+\frac{1}{n^{\frac{3}{2}}}\sum_{i=1}^{n}\sum_{k\leq h}\int_{0}^{t}f^{\prime}(\Phi_{i}(Y^{(n)}(w)))\frac{\partial \Phi_{i}}{\partial y_{k,l}}(Y^{(n)}(w))\delta X_{k,h}(w)\\
	&\hspace{.5cm}+\frac{1}{2}\int_{0}^{t}\int_{\R^2}\Indi{\{x\neq y\}}\frac{f^{\prime}(x)-f^{\prime}(y)}{x-y}\mu_{w}^{(n)}(\ud x)\mu_{w}^{(n)}(\ud y)v_{w}^{\prime}\ud w+\frac{1}{2n^{2}}\sum_{i=1}^{n}\int_{0}^{t}f^{\prime\prime}(\Phi_{i}(Y^{(n)}(w)))v_{w}^{\prime}\ud w,
\end{split}
\end{equation}
where $v_{w}:=\sigma_{w}^2$ and the Skorohod integration is understood in the generalized sense. Recalling the definitions of  $Z_{f}^{(n)}$ and $\Psi_{k,h}$, which are given in \eqref{eq:Ztfn} and \eqref{eq:psik1def} respectively, using the previous equation, as well as the identity 
\begin{align}\label{eq:differencequotientint}
\frac{f^\prime(x)-f^{\prime}(y)}{x-y}=\int_{0}^{1}f^{\prime\prime}(\theta x+(1-\theta)y)\ud\theta,
\end{align} 
 we  deduce 
\begin{align}\label{eq:Xdecomptight}
Z_{f}^{(n)}(t)
  &=\delta^*(u_{f,t}^{(n)})+G_{f,t}^{(n)},
	\end{align}
where 
\begin{align}
G_{f,t}^{(n)}
  &:=\frac{1}{2n}\int_{0}^{t}\sum_{1\leq i_{1},i_{2}\leq n}\int_{0}^1\bigg(f^{\prime\prime}\Big(\theta\Phi_{i_{1}}(Y^{(n)}(w))+(1-\theta)\Phi_{i_{2}}(Y^{(n)}(w))\Big)\nonumber\\
	&\ \ \ \ \ \ \ \ \ \ \ \ \ \ \ \ \ -\E\left[f^{\prime\prime}\Big(\theta\Phi_{i_{1}}(Y^{(n)}(w))+(1-\theta)\Phi_{i_{2}}(Y^{(n)}(w))\Big)\right]\bigg)v_{w}^{\prime}\ud\theta \ud w,\label{eq:Gdef}
\end{align}
and $u_{f,t}^{(n)}=(u_{f,t}^{(n)}(w), w\geq0)\in {\rm L}^{\beta}([0,T],\R^{d(n)})$ is the $\R^{d(n)}$-valued process defined by
\[
u_{f,t}^{(n)}(w):=\left\{n^{-\frac{1}{2}}\Psi_{k,h}[f^{\prime}](Y^{(n)}(w)); 1\leq k\leq h\leq  n\right\},\quad \mathrm{for} \quad w\in[0,t],  
\]
and $u_{f,t}^{(n)}(w):=0$ otherwise. In order to prove our result, it suffices to show that for all $T>0$, the processes $(\delta^*(u_{f, t}^{(n)}), t\geq 0)$ and $ (G_{f, t}^{(n)}, t\geq 0)$ are tight in $\Cc[0,T]$. Since $\delta^*(u_{f, 0}^{(n)})=G_{f, 0}^{(n)}=0$, by  Billingsley´s criterion \cite[Theorem~12.3]{billingsley2013convergence}, it is enough to show that for $i=1,2$ there exist  $C>0$, such that for all $\gamma>1$,
\begin{align}
\E\left[\Abs{\delta^*(u_{f,t}^{(n)})-\delta^*(u_{f,s}^{(n)})}^{2\gamma}\right]
 &\leq C\Abs{t-s}^{\frac{\gamma}{\beta}}\label{eq:BillingsleyZ}\\
\E\left[\Abs{G_{f,t}^{(n)}-G_{f,s}^{(n)}}^{2\gamma}\right]
 &\leq C\Abs{t-s}^{2\varepsilon\gamma}\label{eq:BillingsleyG}
\end{align}
where $\beta=\frac{\alpha}{\alpha-1}$, for $\alpha$ given as in  \textbf{(H1)} and $\varepsilon$ is as in  \textbf{(H2)}.

For simplicity on exposition, we divide the rest proof in two steps which correspond to each of the previous inequalities.

\textbf{Inequality \eqref{eq:BillingsleyZ}.} For $s,t>0$, $n\in\N$ and $f\in \mathcal{P}$ fixed, we introduce  
\begin{equation}\label{eq:XGZKdef}
\Delta Z^{(n)} :=Z_{f}^{(n)}(t)-Z_{f}^{(n)}(s),\quad
\Delta u^{(n)}
  :=u_{f,t}^{(n)}-u_{f,s}^{(n)}\quad\textrm{and}\quad
\Delta G^{(n)}
  :=G_{f,t}^{(n)}-G_{f,s}^{(n)}.
\end{equation}
In particular,  we observe
\begin{align}
\Delta u^{(n)}(y)
  =\left\{n^{-\frac{1}{2}}\Psi_{k,h}[f^{\prime}](Y^{(n)}(y)); 1\leq k\leq h\leq  n\right\},\quad \mathrm{for} \quad y\in (s,t],  \label{eq:deltaudef}
\end{align}
and $\Delta u^{(n)}(y)=0$ otherwise. Our goal is to find an upper bound for $\E\left[(\delta^*(\Delta u^{(n)}))^{2\gamma}\right]$ for every $\gamma>\beta$. By H\"older inequality, we deduce
\[
\begin{split}
\E\left[(\delta^*(\Delta u^{(n)}))^{2\gamma}\right]&=\E\left[(\delta^*(\Delta u^{(n)}))^{2\gamma-1}\delta^*(\Delta u^{(n)})\right]\\
	&=(2\gamma-1)\E\left[(\delta^*(\Delta u^{(n)}))^{2\gamma-2}\Ip{D\delta^*(\Delta u^{(n)}),\Delta u^{(n)}}_{\Hg^{d(n)}}\right]\\
  &\leq (2\gamma-1)\E\left[(\delta^*(\Delta u^{(n)}))^{2\gamma}\right]^{1-\frac{1}{\gamma}}\left\|\Ip{D\delta^*(\Delta u^{(n)}),\Delta u^{(n)}}_{\Hg^{d(n)}}\right\|_{{\rm L}^{\gamma}(\Omega)}.
\end{split}
\]
where the second equality follows from \eqref{eq:dualitydeltaD}.  From the previous identity, it follows  
\begin{align}\label{ineq:deltamallmat}
\left\|(\delta^*(\Delta u^{(n)}))^{2}\right\|_{{\rm L}^{\gamma}(\Omega)}
  &\leq (2\gamma-1)\Norm{\Ip{D\delta^*(\Delta u^{(n)}),\Delta u^{(n)}}_{\Hg^{d(n)}}}_{{\rm L}^{\gamma}(\Omega)},
\end{align}
By \eqref{eq:Xdecomptight}, we deduce  $\delta^*(\Delta u^{(n)})=\Delta Z^{(n)}-\Delta G^{(n)}$. Hence, using \eqref{ineq:deltamallmat}, we get 

\begin{equation}\label{ineq:deltasupuXGK}
\begin{split}
\left\|(\delta^*(\Delta u^{(n)}))^2\right\|_{{\rm L}^{\gamma}(\Omega)}
  &\leq 2(2\gamma-1)\Bigg(\sup_{w\in[0,T]}\Norm{\Ip{DZ_{f}^{(n)}(w),\Delta u^{(n)}}_{\Hg^{d(n)}}}_{{\rm L}^{\gamma}(\Omega)}\\
	&\hspace{4cm} +\sup_{w\in[0,T]}\Norm{\Ip{DG_{f,w}^{(n)},\Delta u^{(n)}}_{\Hg^{d(n)}}}_{{\rm L}^{\gamma}(\Omega)}\bigg).
\end{split}
\end{equation}
Thus, it is enough to upper bound the two terms appearing in the right-hand side of the previous inequality. To upper bound the first term, we  recall the definition of $DZ_{f}^{(n)}$ in \eqref{DZ_f} and observe from  \eqref{eq:deltaudef}, the following
\[
\begin{split}
&\Norm{\Ip{DZ_{f}^{(n)}(w),\Delta u^{(n)}}_{\Hg^{d(n)}}}_{{\rm L}^{\gamma}(\Omega)}\\
  &\hspace{2cm}\leq \bigg\|\frac{1}{n}\sum_{k\leq h}\int_{s}^{t}\Psi_{k,h}[f^{\prime}](Y^{(n)}(w))\Psi_{k,h}[f^{\prime}](Y^{(n)}(x))\frac{\partial R}{\partial x}(x,w)\ud x\bigg\|_{{\rm L}^{\gamma}(\Omega)}\\
	&\hspace{2cm}\leq \frac{1}{n}\int_{s}^{t}\bigg\|\sum_{k\leq h}\Psi_{k,h}[f^{\prime}](Y^{(n)}(w))\Psi_{k,h}[f^{\prime}](Y^{(n)}(x))\bigg\|_{{\rm L}^{\gamma}(\Omega)}\Abs{\frac{\partial R}{\partial x}(x,w)}\ud x.
\end{split}
\]
Let $A(n)$ be a standard Gaussian orthogonal ensemble and define $M_T:=\sup_{0\leq t\leq T}\sigma_{t}$. Using the Cauchy-Schwarz inequality (twice), we get that 
\begin{multline*}
\left\|\sum_{k\leq h}\Psi_{k,h}[f^{\prime}](Y^{(n)}(w))\Psi_{k,h}[f^{\prime}](Y^{(n)}(x))\right\|_{{\rm L}^{\gamma}(\Omega)}\\
\begin{aligned}
  &\leq\left\|\sum_{k\leq h}\left(\Psi_{k,h}[f^{\prime}](Y^{(n)}(w))\right)^2\right\|_{{\rm L}^{\gamma}(\Omega)}^{\frac{1}{2}}
	\left\|\sum_{k\leq h}\Big(\Psi_{k,h}[f^{\prime}](Y^{(n)}(x))\Big)^2\right\|_{{\rm L}^{\gamma}(\Omega)}^{\frac{1}{2}}\\
	&\leq \sup_{0\leq z\leq M_T}\left\|\sum_{k\leq h}\Big(\Psi_{k,h}[f^{\prime}](zA(n))\Big)^2\right\|_{{\rm L}^{\gamma}(\Omega)}\\
&	=2\sup_{0\leq z\leq M_T}\left\|\sum_{i=1}^n\Big(f^{\prime}(\Phi_{i}(zA(n)))\Big)^2\right\|_{{\rm L}^{\gamma}(\Omega)},
\end{aligned}
\end{multline*}
where the last identity follows from \eqref{eq:isimetrypsik1}. Using the previous inequality, as well as Lemma \ref{lem:Wignernorm}, we deduce that there exists a constant $C>0$, only depending on $f$, $\gamma,\sigma$ and $T$, such that 
\begin{align}\label{eq:DXFwuaux}
\Norm{\Ip{DZ_{f}^{(n)}(w),u}_{\Hg^{d(n)}}}_{{\rm L}^{\gamma}(\Omega)}
  &\leq C\int_{s}^{t}\Abs{\frac{\partial R}{\partial x}(x,w)}{\rm d} x.
\end{align}
Hence, by H\"older inequality and condition \textbf{(H1)}, we get 
\begin{equation}\label{eq:DXFwu}
\begin{split}
\int_{s}^{t}\Abs{\frac{\partial R}{\partial x}(x,w)}{\rm d} x
	&\leq C\Abs{t-s}^{\frac{1}{\beta}}\bigg(\int_{s}^{t}\Abs{\frac{\partial R}{\partial x}(x,w)}^{\alpha}{\rm d} x\bigg)^{\frac{1}{\alpha}}\\
	&\leq  C\Abs{t-s}^{\frac{1}{\beta}}\bigg(\sup_{w\in[0,T]}\int_{0}^{T}\Abs{\frac{\partial R}{\partial x}(x,w)}^{\alpha}{\rm d} x\bigg)^{\frac{1}{\alpha}}.
\end{split}
\end{equation}
Combining \eqref{eq:DXFwuaux} and \eqref{eq:DXFwu}, we deduce that there exists a constant $C_1>0$, independent of $s,t,w$ and $n$, such that 
\begin{align}\label{ineq:boundforXfvsu}
\Norm{\Ip{DZ_{f}^{(n)}(w),u}_{\Hg^{d(n)}}}_{{\rm L}^{\gamma}(\Omega)}
	&\leq  C_1\Abs{t-s}^{\frac{1}{\beta}},
\end{align}
which gives the desired bound for the first term in \eqref{ineq:deltasupuXGK}.

In order to upper bound $\Norm{\Ip{DG_{f,w}^{(n)},\Delta u}_{\Hg^{d(n)}}}_{{\rm L}^{\gamma}(\Omega)}$ we follow a similar approach as above. To simplify the notation, we introduce
\[
\mathfrak{F}_{i,j}(n,\theta, w):= f^{\prime\prime\prime}\Big(\theta\Phi_{i}(Y^{(n)}(w))+(1-\theta)\Phi_{j}(Y^{(n)}(w))\Big),
\]
and
\[
\mathfrak{I}_{k,h}^{i,j}(n,\theta,w):=\theta\frac{\partial\Phi_{i}}{\partial x_{k,h}}(Y^{(n)}(w))+(1-\theta)\frac{\partial\Phi_{j}}{\partial x_{k,h}}(Y^{(n)}(w)).
\]
Next, we observe that $DG_{f,w}^{(n)}=\{v_{k,h}; 1\leq k\leq h\leq n\}$ with
\begin{equation}\label{eq:DGform}
\begin{split}
v_{k,h}(\cdot)
  &=\frac{1}{2n^{\frac{3}{2}}}\int_{0}^{w}\int_{0}^1\sum_{1\leq i_{1}, i_{2}\leq n}
  \mathfrak{F}_{i_1,i_2}(n, \theta, r)\mathfrak{I}_{k,h}^{i_1,i_2}(n,\theta, r)
  v_{r}^{\prime}\Indi{[0,r]}(\cdot){\rm  d}\theta  {\rm d}r.
\end{split}
\end{equation}
Thus, by \eqref{eq:deltaudef}, we deduce
\begin{equation}
\begin{split}\label{eq:IpDGFaux1}
&\left|\Ip{DG_{f,w}^{(n)},\Delta u}_{\Hg^{d(n)}}\right|\\
  &\hspace{1cm}\leq \bigg|\frac{1}{2n^{2}}\int_{0}^{w}\int_{0}^1\int_{s}^{t}\sum_{1\leq i_{1}, i_{2}\leq n}\mathfrak{F}_{i_1,i_2}(n, \theta, r) \sum_{k\leq h}\mathfrak{I}_{k,h}^{i_1,i_2}(n,\theta, r)\\
	&\hspace{7.5cm}\times\Psi_{k,h}[f^{\prime}](Y^{(n)}(x))
	\frac{\partial R}{\partial x}(x,r)v_r^{\prime}\ud x\ud\theta \ud r\bigg|.
\end{split}
\end{equation}
Next we find suitable bounds for the summands appearing in the right hand side. To this end, we define 
\begin{align*}
\mathcal{T}^{(n)}_{k,h}(\theta,r)
     &:=\sum_{1\leq i_{1}, i_{2}\leq n}\mathfrak{F}_{i_1,i_2}(n, \theta, r) \mathfrak{I}_{k,h}^{i_1,i_2}(n,\theta, r),\quad \textrm{and}\quad\mathcal{S}^{(n)}_{k,h}(x):=\Psi_{k,h}[f^{\prime}](Y^{(n)}(x)),
\end{align*}
so that  inequality \eqref{eq:IpDGFaux1} can be rewritten as follows
\begin{align}\label{eq:IpDGFaux2}
\left|\Ip{DG_{f,w}^{(n)},\Delta u}_{\Hg^{d(n)}}\right|
  &\leq \frac{1}{2n^{2}}\int_{0}^{w}\int_{0}^1\int_{s}^{t}\bigg|\sum_{k\leq h}\mathcal{T}^{(n)}_{k,h}(\theta,r)\mathcal{S}^{(n)}_{k,h}(x)\frac{\partial R}{\partial x}(x,r)v_r^{\prime}\bigg|\ud x\ud\theta \ud r.
\end{align}
Using Minkowski  and Cauchy  inequalities in \eqref{eq:IpDGFaux1}, we deduce
\begin{multline}\label{eq:IpDGFaux12}
\left\|\Ip{DG_{f,w}^{(n)},\Delta u}_{\Hg^{d(n)}}\right\|_{{\rm L}^{\gamma}(\Omega)}\\
\begin{aligned}
  &\leq \frac{1}{2n^{2}}\int_{0}^{w}\int_{0}^1\int_{s}^{t}\bigg\|\bigg(\sum_{k\leq h}\mathcal{T}^{(n)}_{k,h}(\theta,r)^2\bigg)^{\frac{1}{2}}\bigg(\sum_{k\leq h}\mathcal{S}^{(n)}_{k,h}(x)\bigg)^{\frac{1}{2}}\bigg\|_{{\rm L}^{\gamma}(\Omega)}\bigg|\frac{\partial R}{\partial x}(x,r)v_r^{\prime}\bigg|\ud x\ud\theta \ud r\\
  &\leq \frac{1}{2n^{2}}\int_{0}^{w}\int_{0}^1\int_{s}^{t} \bigg\|\sum_{k\leq h}\mathcal{T}^{(n)}_{k,h}(\theta,r)^2\bigg\|_{{\rm L}^{\gamma}(\Omega)}^{\frac{1}{2}}\bigg\|\sum_{k\leq h}\mathcal{S}^{(n)}_{k,h}(x)^2\bigg\|_{{\rm L}^{\gamma}(\Omega)}^{\frac{1}{2}}\bigg|\frac{\partial R}{\partial x}(x,r)v_r^{\prime}\bigg|\ud x\ud\theta \ud r.
\end{aligned}
\end{multline}
 From identity \eqref{eq:D1PhiUpsum}, it follows that 
\begin{equation}\label{eq:IpDGFauxv2}
\begin{split}
    \sum_{k\leq h}\mathcal{T}^{(n)}_{k,h}(\theta,r)^2
      &=4\sum_{1\leq i_1,i_2,i_3,i_4\leq n}\mathfrak{F}_{i_1,i_2}(n, \theta, r) \mathfrak{F}_{i_3,i_4}(n, \theta, r)\\
      &\hspace{3cm}\times\bigg(\theta^2\delta_{i_1,i_3}+\theta(1-\theta)(\delta_{i_1,i_4}+\delta_{i_2,i_3})+ (1-\theta)^2\delta_{i_2,i_4}\bigg). 
\end{split}
\end{equation}
Since $f\in\Pc$, there exist constants $C_2>0$ and $a\in\N$, such that $|f^{\prime\prime\prime}(x)|\leq  C_2(1+|x|^{2a})$. Applying this inequality in \eqref{eq:IpDGFauxv2}, and using the fact that $0\leq \theta\leq 1$, we get
\begin{equation*}
    \begin{split}
\sum_{k\leq h}\mathcal{T}^{(n)}_{k,h}(\theta,r)^2
  &\leq  C_2\sum_{1\leq i_{1} ,i_{2},i_3,i_4\leq n}\bigg(1+\sum_{\ell=1}^4\Big(\Phi_{i_{\ell}}(Y^{(n)}(r))\Big)^{4a}\bigg)\bigg( \delta_{i_1,i_3}+ \delta_{i_1,i_4}+ \delta_{i_2,i_3}+  \delta_{i_2,i_4}\bigg)\\
  &\leq  8\times C_2\sum_{1\leq i_{1} ,i_{2},i_3\leq n}\bigg(1+\sum_{\ell=1}^3\Big(\Phi_{i_{\ell}}(Y^{(n)}(r))\Big)^{4a}\bigg)\\
  &\le  16\times C_2n^2\sum_{i=1}^n\bigg(1+\Big(\Phi_{i}(Y^{(n)}(r))\Big)^{4a}\bigg).
\end{split}
\end{equation*}
Therefore, by Lemma \ref{lem:Wignernorm},
\begin{align}
\bigg\|\sum_{k\leq h}\mathcal{T}^{(n)}_{k,h}(\theta,r)^2\bigg\|_{{\rm L}^{\gamma}(\Omega)}
  &\leq \tilde{C}_2n^3,\label{eq:IpGdu2case}
\end{align}
for some constant $\tilde{C}_2>0$ independent of $\theta,r$ and $n$.

On the other hand, from the definition of $\Psi_{k,h}$ (see \eqref{eq:psik1def}),  and identity \eqref{eq:D1PhiUpsum}, we get  
\begin{align*}
    \sum_{k\leq h}\mathcal{S}^{(n)}_{k,h}(x)^2
    &=2\sum_{i=1}^nf^{\prime}(\Phi_{i}(Y^{(n)}(x)))^2.
\end{align*}
Similarly as above, applying the polynomial growth property of $f^{\prime}$, combined with Lemma \ref{lem:Wignernorm}, we deduce the existence of a constant $C_3>0$, such that 
\begin{align}\label{eq:IpGdu2case2}
\bigg\|\sum_{k\leq h}\mathcal{S}^{(n)}_{k,h}(x)^2\bigg\|_{{\rm L}^{\gamma}(\Omega)}
  &\leq C_3n.
\end{align}
Next, we use  \eqref{eq:IpDGFaux12}, \eqref{eq:IpGdu2case} and \eqref{eq:IpGdu2case2}, to deduce
\begin{align*}
\Big\|\Ip{DG_{f,w}^{(n)},\Delta u}_{\Hg^{d(n)}}\Big\|_{{\rm L}^{\gamma}(\Omega)}
  &\leq C_4\int_{0}^{w}\int_{s}^{t}\left|\frac{\partial R}{\partial x}(x,r)v_r^{\prime}\right|\ud x \ud r,
\end{align*}
which by \eqref{eq:DXFwu} and hypothesis \textbf{(H2)} , implies that 
\begin{align}\label{ineq:boundforXfvsuc2}
\Big\|\Ip{DG_{f,w}^{(n)},\Delta u}_{\Hg^{d(n)}}\Big\|_{{\rm L}^{\gamma}(\Omega)}
	&\leq  C_5\Abs{t-s}^{\frac{1}{\beta}},
\end{align}
for some constant $C_5>0$, which is  independent of $n$ and $w$. This gives the desired bound for the second term in \eqref{ineq:deltasupuXGK}.

\textbf{Inequality \eqref{eq:BillingsleyG}}.  By inequality \eqref{eq:centeredLpbound}, we have that for all $\gamma>1$, 
\begin{align}\label{eq:GdfincDandD2}
\Big\|G_{f,t}^{(n)}-G_{f,s}^{(n)}\Big\|_{{\rm L}^{\gamma}(\Omega)}
  \leq \Big\|DG_{f,t}^{(n)}-DG_{f,s}^{(n)}\Big\|_{{\rm L}^{\gamma}(\Omega;\Hg^{d(n)})}+\Big\|D^2G_{f,t}^{(n)}-D^2G_{f,s}^{(n)}\Big\|_{{\rm L}^{\gamma}(\Omega;(\Hg^{d(n)})^2)}.
\end{align}
To bound the first term in the right hand side, we proceed as follows. Recall  
\[
DG_{f,w}^{(n)}=\{v_{k,h}; 1\leq k\leq h\leq n\},
\]
with $v_{k,h}$ given by \eqref{eq:DGform}. Therefore, 
\begin{equation}\label{eq:DGexp1}
\begin{split}
\Big\|&DG_{f,t}^{(n)}-DG_{f,s}^{(n)}\Big\|_{\Hg^{d(n)}}^2\\
  &=\frac{\sigma_{t}^2}{4n^3}\int_{[s,t]^2}\int_{[0,1]^2}\sum_{1\leq i_1,i_2,j_1,j_2\leq n}\mathfrak{F}_{i_1,i_2}(n,\theta_1, w_1)\mathfrak{F}_{j_1,j_2}(n,\theta_2, w_2)\\
	&\hspace{2cm}\times\sum_{k\leq h}\mathfrak{I}_{k,h}^{i_1,i_2}(n,\theta_1, w_1)\mathfrak{I}_{k,h}^{j_1,j_2}(n,\theta_2, w_2) R(w_1,w_2)v_{w_1}^{\prime} v_{w_2}^{\prime}\ud\theta_1\ud\theta_2 \ud w_1\ud w_2.
	\end{split}
\end{equation}
Using \eqref{eq:D1PhiUpsum}, we deduce  
\begin{equation}\label{eq:DGexp2ss}
\begin{split}
\sum_{k\leq h}&\mathfrak{I}_{k,h}^{i_1,i_2}(n,\theta_1, w_1)\mathfrak{I}_{k,h}^{j_1,j_2}(n,\theta_2, w_2)\\
	&\hspace{1cm}=4\Big(\theta_1\theta_2\delta_{i_1,j_1}+\theta_1(1-\theta_2)\delta_{i_1,j_2}+(1-\theta_1)\theta_2\delta_{i_2,j_1}+(1-\theta_1)(1-\theta_2)\delta_{i_2,j_2}\Big)\\
	&\hspace{1cm}\leq 4(\delta_{i_1,j_1}+\delta_{i_1,j_2}+\delta_{i_2,j_1}+\delta_{i_2,j_2}).
\end{split}
\end{equation}
Similarly as before, since $f^{\prime\prime\prime}\in\Pc$ there are  constants $C_2>0$ and $a>1$, such that $|f(x)|\leq C_2(1+|x|^{2a})$, which in turn implies that there exists a constant $C_6>0$, such that $|f(x+y)|\leq C_6(1+|x|^{2a}+|y|^{2a})$. Using this observation, as well as Minkowski inequality and identities \eqref{eq:DGexp1} and \eqref{eq:DGexp2ss}, we get
\[
\begin{split}
\Big\|&DG_{f,t}^{(n)}-DG_{f,s}^{(n)}\Big\|_{{\rm L}^{\gamma}(\Omega;\Hg^{d(n)})}^2\\
&\leq\frac{C_6{T^{2H}}}{n^3}\int_{[s,t]^2}\Bigg\|{\sum_{1\leq i_1,i_2,j_1,j_2\leq n}}
	\Big(1+\Big(\Phi_{i_{1}}(Y^{(n)}(w_1))\Big)^{2a}+\Big(\Phi_{i_{2}}(Y^{(n)}(w_1))\Big)^{2a}\Big)\\
	&\hspace{5.6cm}\times \Big(1+\Big(\Phi_{j_{1}}(Y^{(n)}(w_2))\Big)^{2a}+\Big(\Phi_{j_{2}}(Y^{(n)}(w_2))\Big)^{2a}\Big)\\
	&\hspace{6.5cm}\times|v_{w_1}^{\prime} v_{w_2}^{\prime}|\Big(\delta_{i_1,j_1}+\delta_{i_1,j_2}+\delta_{i_2,j_1} +\delta_{i_2,j_2}\Big)\Bigg\|_{{\rm L}^{\gamma}(\Omega)}\ud w_1\ud w_2\\
	&\leq\frac{18C_6{T^{2H}}}{n^3}\int_{[s,t]^2}\Bigg\|{\sum_{1\leq i_1,i_2,j_1,j_2\leq n}}
	\Big(1+\sum_{\ell=1}^2\Big(\Phi_{i_{1}}(Y^{(n)}(w_1))\Big)^{4a}+\sum_{\ell=1}^2\Big(\Phi_{j_{\ell}}(Y^{(n)}(w_2))\Big)^{4a}\Big)\\
	&\hspace{6.5cm}\times| v_{w_1}^{\prime} v_{w_2}^{\prime}|\Big(\delta_{i_1,j_1}+\delta_{i_1,j_2}+\delta_{i_2,j_1} +\delta_{i_2,j_2}\Big)\Bigg\|_{{\rm L}^{\gamma}(\Omega)}
 \ud w_1\ud w_2.
\end{split}
\]
We proceed similarly as in \eqref{eq:IpGdu2case} to deduce 
\begin{align*}
\Big\|DG_{f,t}^{(n)}-DG_{f,s}^{(n)}\Big\|_{{\rm L}^{\gamma}(\Omega;\Hg^{d(n)})}^2
  \leq C_{7}\int_{[s,t]^2}|v_{w_1}^{\prime} v_{w_2}^{\prime}|\ud w_1\ud w_2,
\end{align*}
for some constant $C_{7}>0$, that depends only on $a$ and  $\sup_{0\leq w\leq T}\sigma_{w}$. Next, we use  condition \textbf{(H2)} to get  
\begin{align}\label{eq:BillingsleyGprev}
\Big\|DG_{f,t}^{(n)}-DG_{f,s}^{(n)}\Big\|_{{\rm L}^{\gamma}(\Omega;\Hg^{d(n)})}^2
  \leq C_{8}(t^{\varepsilon}-s^{\varepsilon})^2,
\end{align}
for $C_{8}>0$. Since $\varepsilon\in(0,1)$, we have  $|t^{\varepsilon}-s^{\varepsilon}|\leq|t-s|^{\varepsilon}$, and thus
\begin{align}\label{eq:BillingsleyGprev2}
\Big\|DG_{f,t}^{(n)}-DG_{f,s}^{(n)}\Big\|_{{\rm L}^{\gamma}(\Omega;\Hg^{d(n)})}
  \leq C_{9}|t-s|^{\varepsilon},
\end{align}
for $C_{9}>0$, 
which gives a bound for the first term in the right-hand side of \eqref{eq:GdfincDandD2}.  To handle the second term in \eqref{eq:GdfincDandD2}, we follow a similar approach but we remark that the computations are longer due to the appearance of terms involving the second derivatives of the functions $\Phi_{i}$, with $i\in\{1,\dots, n\}$. We first observe from \eqref{eq:Gdef}, that in order to compute $D^{2}\Delta G ^{(n)}$, the knowledge of the second Malliavin derivative of variables of the form $f^{\prime\prime}(\theta\Phi_{i_1}(Y^{(n)}(w))+(1-\theta) \Phi_{i_2}(Y^{(n)}(w)))$, for $w\leq T$ and $\theta\in[0,1]$, are necessary. To this end, we introduce the notation 
\[
\mathfrak{K}_{i,j}(n, \theta, w)=f^{(4)}\Big(\theta\Phi_{i}(Y^{(n)}(w))+(1-\theta)\Phi_{j}(Y^{(n)}(w))\Big),
\]
and
\[
D^2f^{\prime\prime}(\theta\Phi_{i_1}( Y^{(n)}(w))+(1-\theta) \Phi_{i_2}(Y^{(n)}(w)))\\
  :=\Big[\zeta_{k,h}^{p,q}(n,i_1,i_2,\theta, w)\Big]_ {\substack{1\leq k\le h\leq n \\ 1\leq p\leq q\leq n}},
\]
where

\[
\begin{split}
\zeta_{k,h}^{p,q}(n,&i_1,i_2,\theta, w)
  =\frac{1}{n}\mathfrak{K}_{i_1,i_2}(n, \theta, w)\mathfrak{I}_{k,h}^{i_1,i_2}(n,\theta,w)\mathfrak{I}_{k,h}^{i_1,i_2}(n,\theta,w)\Indi{[0,w]}^{\otimes 2}\\
	&+\frac{1}{n}\mathfrak{F}_{i_1,i_2}(n,\theta, w)\left(\theta\frac{\partial^2\Phi_{i_1}}{\partial y_{k,h}\partial y_{p,q}}(Y^{(n)}(w))+(1-\theta)\frac{\partial^2\Phi_{i_2}}{\partial y_{k,h}\partial y_{p,q}}(Y^{(n)}(w))\right)\Indi{[0,w]}^{\otimes 2}.
\end{split}
\]

Next, by using \eqref{eq:D1PhiUpc} and \eqref{eq:Pikhpqdefprev}, as well as identity \eqref{eq:differencequotientint}, we have  

\begin{align*}
\sum_{i_1,i_2=1}^nD^2f^{\prime\prime}(\theta\Phi_{i_1}( Y^{(n)}(w))+(1-\theta) \Phi_{i_2}(Y^{(n)}(w)))
  &=\Theta(1,w)+\Theta(2,w)+\Theta(3,w),
\end{align*}
where $\Theta(\ell,w)=\{\Theta_{k,h}^{p,q}(\ell,w)\ ;\ 1\leq k\leq h\leq n\ \text{ and }\ 1\leq p\leq q\leq n\ \}$, for $\ell=1,2,3$ are given by
\begin{align*}
\Theta_{k,h}^{p,q}(1,w)
  &:=\frac{1}{n}\sum_{i_1,i_2=1}^{n}\mathfrak{K}_{i_1,i_2}(n, \theta, w)\big(\theta V_{k,h}^{i_1,i_1}(Y^{(n)}(w))+(1-\theta)V_{k,h}^{i_2,i_2}(Y^{(n)}(w))\big)\\
	&\hspace{4cm}\times\big(\theta V_{p,q}^{i_1,i_1}(Y^{(n)}(w))+(1-\theta)V_{p,q}^{i_2,i_2}(Y^{(n)}(w))\big)\Indi{[0,w]}^{\otimes 2},\\
	\end{align*}

	\begin{align*}
\Theta_{k,h}^{p,q}(2,w)
	&:=\frac{\theta^2}{n}\sum_{\substack{1\leq  i_1, i_2,i_3\leq n\\i_1\neq i_3}}\int_{0}^1
	f^{(4)}\bigg(\vartheta\theta\Phi_{i_1}(Y^{(n)}(w))+(1-\vartheta)\theta\Phi_{i_3}(Y^{(n)}(w))\\
	&\ \ \ \ \ \ \ \ \ \ \ \ \ \  +(1-\theta)\Phi_{i_2}(Y^{(n)}(w)) \bigg)
	V_{k,h}^{i_1,i_3}(Y^{(n)}(w))V_{p,q}^{i_1,i_3}(Y^{(n)}(w))\text{d}\vartheta\Indi{[0,w]}^{\otimes 2},\\
		\end{align*}
		 and
			\begin{align*}
\Theta_{k,h}^{p,q}(3,w)
	&:=\frac{(1-\theta)^2}{n}\sum_{\substack{1\leq  i_1, i_2,i_3\leq n\\i_2\neq i_3}}\int_{0}^1
	f^{(4)}\bigg(\theta\Phi_{i_1}(Y^{(n)}(w))+\vartheta (1-\theta)\Phi_{i_2}(Y^{(n)}(w))\\
	&\ \ \ \ \ \ \ \ \ \    + (1-\vartheta)(1-\theta)\Phi_{i_3}(Y^{(n)}(w))\bigg)
	V_{k,h}^{i_2,i_3}(Y^{(n)}(w))V_{p,q}^{i_2,i_3}(Y^{(n)}(w))\text{d}\vartheta\Indi{[0,w]}^{\otimes 2}.
\end{align*}
On the other hand, by applying Minkowski's inequality, as well as the definition of $G_{f,w}^{(n)}$, which is given by \eqref{eq:Gdef}, we deduce 
\begin{multline}\label{eq:D2Gftinc}
\Big\|D^{2} G_{f,t}^{(n)}-D^{2} G_{f,s}^{(n)}\Big\|_{{\rm L}^{2\gamma}(\Omega)}\\
\begin{aligned}
  &\leq 	\frac{1}{2n^2}\int_{s}^{t}\int_{0}^1\Big\|D^2f^{\prime\prime}(\theta\Phi_{i_1}( Y^{(n)}(w))+(1-\theta) \Phi_{i_2}(Y^{(n)}(w)))\Big\|_{{\rm L}^{2\gamma}(\Omega;(\Hg^{d})^{\otimes 2})}|v_{w}^{\prime}|\ud\theta \ud w\\
  &\leq 	\frac{1}{2n^2}\int_{s}^{t}\int_{0}^1\sum_{\ell=1,2,3} \|\Theta (\ell, w)\|_{{\rm L}^{2\gamma}(\Omega;(\Hg^{d})^{\otimes 2})}|v_{w}^{\prime}|\ud\theta \ud w.
\end{aligned}
\end{multline}
Next we bound the terms $\|\Theta (\ell, w)\|_{{\rm L}^{2\gamma}(\Omega;(\Hg^{d})^{\otimes 2})}$, for $\ell=1,2,3$. In order to handle  the case $\ell=1$, we first notice that by \eqref{eq:VandW}, for all $1\leq i_1,i_2,j_2,j_2\leq n$, 
\[
\begin{split}
\sum_{\substack{1\leq k\leq h\leq n\\1\leq p\leq q\leq n}}&\big(\theta V_{k,h}^{i_1,i_1} +(1-\theta)V_{k,h}^{i_2,i_2} \big)\big(\theta V_{p,q}^{i_1,i_1} +(1-\theta)V_{p,q}^{i_2,i_2} \big)\\
&\hspace{3cm}\times\big(\theta V_{k,h}^{j_1,j_1} +(1-\theta)V_{k,h}^{j_2,j_2} \big)\big(\theta V_{p,q}^{j_1,j_1} +(1-\theta)V_{p,q}^{j_2,j_2} \big)\\
 & =4\big(\theta\delta_{i_1,j_1}+\theta(1-\theta)\delta_{i_1,j_2}+\theta(1-\theta)\delta_{i_2,j_1}+(1-\theta)^2\delta_{i_2,j_2}\big)^2\\
	&\leq32\big(\delta_{i_1,j_1}+ \delta_{i_1,j_2}+ \delta_{i_2,j_1}+ \delta_{i_2,j_2}\big).
\end{split}
\]
Putting all pieces together, we have
\begin{align*}
\|\Theta (1, w)\|_{(\Hg^{d})^{\otimes 2}}^2
  &\leq\frac{32 T^{4H}}{n^2}\sum_{i_1,i_2,j_2=1}^{n}|\mathfrak{K}_{i_1,i_2}(n, \theta, w)||\mathfrak{K}_{i_1,j_2}(n, \theta, w)|\\
&+\frac{32 T^{4H}}{n^2}\sum_{i_1,i_2,j_1=1}^{n}|\mathfrak{K}_{i_1,i_2}(n, \theta, w)| |\mathfrak{K}_{j_1,i_1}(n, \theta, w)|\\
	&+\frac{32T^{4H}}{n^2}\sum_{i_1,i_2,j_2=1}^{n}|\mathfrak{K}_{i_1,i_2}(n, \theta, w)| |\mathfrak{K}_{i_2,j_2}(n, \theta, w)|\\
	&+\frac{32T^{4H}}{n^2}\sum_{i_1,i_2,j_1=1}^{n}|\mathfrak{K}_{i_1,i_2}(n, \theta, w)||\mathfrak{K}_{j_1,i_2}(n, \theta, w)|.\\
\end{align*}
Using the fact that $f^{(4)}$ has polynomial growth and $\theta\in[0,1]$, we can easily deduce from the previous inequality that there exists $a\in\N$, and a constant $C_{10}>0$, than only depends on $f$, such that 
\begin{align*}
\|\Theta (1,w)\|_{(\Hg^{d})^{\otimes 2}}^2
&\leq C_{10}\frac{T^{4H}}{n^2}\sum_{i_1,i_2,i_3=1}^{n}\Big(1+|\Phi_{i_1}(Y^{(n)}(w))|^a+|\Phi_{i_2}(Y^{(n)}(w))|^a+|\Phi_{i_3}(Y^{(n)}(w))|^a\Big), 
\end{align*}
which by Lemma \ref{lem:Wignernorm}, implies that there exist a constant $C_{11}>0$, such that
\begin{align}\label{eq:Thetaw1norm}
\|\Theta (1,w)\|_{{\rm L}^{2\gamma}(\Omega,(\Hg^{d})^{\otimes 2})}^2&\leq C_{10}\frac{T^{4H}}{n^2}\bigg\|\sum_{i_1,i_2,i_3=1}^{n}\bigg(1+|\Phi_{i_1}(Y^{(n)}(w))|^a\nonumber\\
	&\ \ \ \ \ \ \ \ \ \ \ \ \ \ +|\Phi_{i_2}(Y^{(n)}(w))|^a+|\Phi_{i_3}(Y^{(n)}(w))|^a\bigg)\bigg\|_{{\rm L}^{\gamma}(\Omega,(\Hg^{d})^{\otimes 2})}\\
	&\leq C_{11} nT^{4H}.\nonumber
\end{align}
On the other hand, by \eqref{eq:VandW}, for all indices $1\leq d_1,d_2,l_1,l_3\leq n$, we deduce
\begin{align*}
\sum_{\substack{1\leq k\leq h\leq n\\1\leq p\leq q\leq n}}V_{k ,h }^{d_1,l_1}V_{p ,q }^{d_1,l_1}
V_{k,h }^{d_2,l_2}V_{p ,q }^{d_2,l_2}
  &=(\delta_{d_1,d_2}\delta_{l_1,l_2}+\delta_{d_1,l_2}\delta_{l_1,d_2})^2
	\leq 4(\delta_{d_1,d_2}\delta_{l_1,l_2}+\delta_{d_1,l_2}\delta_{l_1,d_2})
\end{align*}
which by an analogous argument  to the proof of \eqref{eq:Thetaw1norm}, leads to 
\begin{align}\label{eq:Thetaw1norm2}
\|\Theta (\ell,w)\|_{{\rm L}^{2\gamma}(\Omega,(\Hg^{d})^{\otimes 2})}^2
  &\leq C_{12}n^2T^{4H},
\end{align}
where $\ell=2,3$ and $C_{12}$ is a strictly positive constant. Therefore, by \eqref{eq:D2Gftinc}, we obtain
\begin{align*}
\Big\|D^{2} G_{f,t}^{(n)}-D^{2} G_{f,s}^{(n)}\Big\|_{{\rm L}^{2\gamma}(\Omega)}
  &\leq 	C_{13}\frac{1}{n}\int_{s}^{t} |v_{w}^{\prime}| \ud w,
\end{align*}
with $C_{13}>0$. Hence, using the condition \textbf{(H2)}, we obtain
\begin{align}\label{eq:BillingsleyGprev3}
\Big\|D^2 G_{f,t}^{(n)}-D^2 G_{f,s}^{(n)}\Big\|_{{\rm L}^{2\gamma}(\Omega;\Hg^{d(n)})}
  \leq C_{14} (t^{\varepsilon}-s^{\varepsilon})
	\leq C_{14} (t-s)^{\varepsilon},
\end{align}
where $C_{14}>0$. 
Finally, by \eqref{eq:GdfincDandD2}, \eqref{eq:BillingsleyGprev2} and \eqref{eq:BillingsleyGprev3}, we obtain 
\begin{align*}
\Big\| G_{f,t}^{(n)}-G_{f,s}^{(n)}\Big\|_{{\rm L}^{2\gamma}(\Omega)}
  \leq C_{15}|t-s|^{\varepsilon},
\end{align*}
as required. This completes the proof.
\end{proof}

\section{Appendix}
Here, we use the same notation as in Section \ref{sec:eigenvalues}. Recall that $d(n)=n(n+1)/2$ and for  every $x\in\R^{d(n)}$,  $\Phi_i(x)$ denotes the $i$-th largest eigenvalue of the matrix $\widehat{x}$. 
\begin{Lemma}\label{lem:app}
Let $V_{k,h}^{i,j}(x)$ be as in \eqref{eq:Vdef}. Then the first and second order partial derivatives of $\Phi_i(x)$ are given by 
\begin{align}
\frac{\partial \Phi_{i}}{\partial x_{k,h}}(x)
  &=V_{k,h}^{i,i}(x),\label{eq:D1PhiU}\\
\frac{\partial^{2}\Phi_{i}}{\partial x_{k,h}\partial x_{p,q}}(x)
  &=\sum_{j=1}^{n}\frac{2\Indi{\{j\neq i\}}}{\Phi_{i}(x)-\Phi_{j}(x)}V_{k,h}^{i,j}(x)V_{p,q}^{i,j}(x).\label{eq:D2partialPhiU1}
	\end{align}
\end{Lemma}
The previous lemma is a particular case of a more general result. Consider an $n \times n$ real symmetric matrix $A(\theta,\beta)$ which is twice continuously differentiable over the real parameters $\theta$ and $\beta$. Assume that $A(\theta,\beta)$ possesses eigenvalues $\lambda_1(\theta,\beta)>\cdots >\lambda_n(\theta,\beta)$ with orthonormal eigenvectors $U_{1}(\theta,\beta),\dots, U_{n}(\theta,\beta)$ of the form $U_{i}(\theta,\beta)=(U_{1,i}(\theta,\beta),\dots, U_{n,i}(\theta,\beta))^T$, which are continuously differentiable over $\theta$ and $\beta$.

\begin{Lemma}
The following Hadamard variational formulas hold true
\begin{align}
\frac{\partial \lambda_i}{\partial\theta}(\theta,\beta)
  &=U_{i}^{*}(\theta,\beta)\frac{\partial A}{\partial\theta}(\theta,\beta)U_i(\theta,\beta)\label{eq:Hadamard1},\\
\frac{\partial^2 \lambda_i}{\partial\theta\partial\beta}(\theta,\beta)
  &=U_{i}^{*}(\theta,\beta)\frac{\partial^2 A}{\partial \theta\partial\beta}(\theta,\beta)U_{i}(\theta,\beta)\nonumber\\
  &+\sum_{j=1}^{n}\frac{2\Indi{\{j\neq i\}}\big(U_{i}^{*}(\theta,\beta)\frac{\partial A}{\partial \theta}(\theta,\beta) U_{j}(\theta,\beta)\big)\big(U_{j}^{*}(\theta,\beta)\frac{\partial A}{\partial \beta}(\theta,\beta)U_{i}(\theta,\beta)\big)}{\lambda_{i}(\theta,\beta)-\lambda_j(\theta,\beta)}.\label{eq:Hadamard2}
\end{align}
\end{Lemma}

Provided that we prove \eqref{eq:Hadamard1} and \eqref{eq:Hadamard2}, we obtain \eqref{eq:D1PhiU} by taking $\theta=x_{k,h}$, and \eqref{eq:D2partialPhiU1} by taking $\theta=x_{k,h}$ and $\beta=x_{p,q}$. The previous lemma can be found in Section~1.3.4 in the book by Tao \cite{tao2012topics}. For the reader's convenience, we provide its proof.

\begin{proof} For simplicity of exposition, in what follows we omit the dependence on the parameters $\theta$ and $\beta$ of  $A(\theta, \beta)$, $U_{i}(\theta,\beta)$ and $\lambda_{i}(\theta,\beta)$. 

We first deduce identity \eqref{eq:Hadamard1}. By taking the derivative with respect to $\theta$ of  $AU_{i}=\lambda_{i}U_{i}$, we get 
\begin{align}\label{eq:Had1}
\frac{\partial A}{\partial \theta}U_{i}+A\frac{\partial U_{i}}{\partial \theta}
  &=\frac{\partial \lambda_{i}}{\partial \theta} U_{i} + \lambda_{i}\frac{\partial U_{i}}{\partial \theta}.
\end{align}
Multiplying \eqref{eq:Had1} by $U_{i}^{*}$ from the left, and using the fact that $U_i^*A=\lambda_iU_i^*$ and $|U_{i}|^2=1$, we have 
\begin{align}\label{eq:Had1aux1}
U_i^*\frac{\partial A}{\partial \theta} U_{i}+ \lambda_{i}U_i^*\frac{\partial U_{i}}{\partial \theta}
  &=\frac{\partial \lambda_{i}}{\partial \theta} + \lambda_{i}U_{i}^*\frac{\partial U_{i}}{\partial \theta}.
\end{align}
On the other hand, if we take the derivative with respect to $\theta$ of $|U_{i}|^2=1$, we obtain 
\begin{align}\label{eq:UidUiorthog}
U_{i}^*\frac{\partial U_{i}}{\partial \theta}=0,
\end{align}
thus,  putting all pieces together, we deduce
\begin{align*}
U_i^*\frac{\partial A}{\partial \theta} U_{i}
  &=\frac{\partial \lambda_{i}}{\partial \theta},
\end{align*}
as required.

For identity \eqref{eq:Hadamard2}, we first take the derivative with respect to $\beta$ in \eqref{eq:Had1}, and obtain
\begin{align*}
\frac{\partial^2 A}{\partial \theta\partial\beta}U_{i}+\frac{\partial A}{\partial \theta}\frac{\partial U_{i}}{\partial\beta}+\frac{\partial A}{\partial\beta}\frac{\partial U_{i}}{\partial\theta}+A\frac{\partial^2 U_{i}}{\partial \theta\partial\beta}
  &=\frac{\partial^2 \lambda_{i}}{\partial \theta\partial\beta} U_{i}+\frac{\partial \lambda_{i}}{\partial \theta} \frac{\partial U_{i}}{\partial\beta}
	+ \frac{\partial\lambda_{i}}{\partial\beta}\frac{\partial U_{i}}{\partial \theta}+\lambda_{i}\frac{\partial^2 U_{i}}{\partial \theta\partial\beta}.
\end{align*}
Again, we multiply by $U_{i}^*$ from the left and use the identities  $U_i^*A=\lambda_iU_i^*$ and $|U_{i}|^2=1$, to deduce
\begin{align*}
U_{i}^{*}\frac{\partial^2 A}{\partial \theta\partial\beta}U_{i}+U_{i}^{*}\frac{\partial A}{\partial \theta}\frac{\partial U_{i}}{\partial\beta}+U_{i}^{*}\frac{\partial A}{\partial\beta}\frac{\partial U_{i}}{\partial\theta}
  &=\frac{\partial^2 \lambda_{i}}{\partial \theta\partial\beta}
	+ \frac{\partial\lambda_{i}}{\partial\theta}U_i^*\frac{\partial U_{i}}{\partial \beta}
	+\frac{\partial\lambda_{i}}{\partial\beta}U_i^*\frac{\partial U_{i}}{\partial \theta}.
\end{align*}
Next, simplifying the above identity and using \eqref{eq:UidUiorthog}, we get
\begin{align}\label{eq:partiallambdathetabeta}
U_{i}^{*}\frac{\partial^2 A}{\partial \theta\partial\beta}U_{i}+U_{i}^{*}\bigg(\frac{\partial A}{\partial \theta}\frac{\partial U_{i}}{\partial\beta}+\frac{\partial A}{\partial\beta}\frac{\partial U_{i}}{\partial\theta}\bigg)
  &=\frac{\partial^2 \lambda_{i}}{\partial \theta\partial\beta}.
\end{align} 
The term inside the  parenthesis, in the left hand side, can be written by expanding $\frac{\partial U_{i}}{\partial\beta}$ and $\frac{\partial U_{i}}{\partial\theta}$ in terms of the basis $U_{1},\dots, U_{n}$, as follows
\begin{align*}
\frac{\partial A}{\partial \theta}\frac{\partial U_{i}}{\partial\beta}+\frac{\partial A}{\partial\beta}\frac{\partial U_{i}}{\partial\theta}
	&=\sum_{j\neq i}\frac{\partial A}{\partial \theta} U_{j}\left\langle \frac{\partial U_{i}}{\partial\beta},U_j\right\rangle+\sum_{j\neq i}\frac{\partial A}{\partial\beta}U_{j}\left\langle\frac{\partial U_{i}}{\partial\theta},U_{j}\right\rangle\nonumber\\
	&+\frac{\partial A}{\partial \theta} U_{i}\left\langle\frac{\partial U_{i}}{\partial\beta},U_{i}\right\rangle+\frac{\partial A}{\partial\beta}U_{i}\left\langle\frac{\partial U_{i}}{\partial\theta},U_{i}\right\rangle\nonumber.
\end{align*}
Hence, using again  \eqref{eq:UidUiorthog}, we observe
\begin{align}\label{eq:Authetabeta}
\frac{\partial A}{\partial \theta}\frac{\partial U_{i}}{\partial\beta}+\frac{\partial A}{\partial\beta}\frac{\partial U_{i}}{\partial\theta}
		&=\sum_{j\neq i}\frac{\partial A}{\partial \theta} U_{j}\left\langle \frac{\partial U_{i}}{\partial\beta},U_j\right\rangle+\sum_{j\neq i}\frac{\partial A}{\partial\beta}U_{j}\left\langle\frac{\partial U_{i}}{\partial\theta},U_{j}\right\rangle.
\end{align}
The inner products in the right hand side can be computed by multiplying \eqref{eq:Had1} by $U_{j}^{*}$ from the left for $j\neq i$, and using the fact that $\lambda_j  U_j^{*}=U_j^{*}A$, to get
\begin{align*}
U_{j}^*\frac{\partial A}{\partial \theta}U_{i}+\lambda_{j}U_{j}^*\frac{\partial U_{i}}{\partial \theta}
  &=\lambda_{i}U_{j}^*\frac{\partial U_{i}}{\partial \theta},
\end{align*}
which implies that for every $i\neq j$,
\begin{align*}
\left\langle\frac{\partial U_{i}}{\partial\theta},U_{j}\right\rangle
  &=U_{j}^{*}\frac{\partial U_{i}}{\partial \theta}
  =\frac{U_{j}^*\frac{\partial A}{\partial \theta}U_{i}}{\lambda_{i}-\lambda_j}.
\end{align*}
Similarly, we have that
\begin{align*}
\left\langle\frac{\partial U_{i}}{\partial\beta},U_{j}\right\rangle 
 &=\frac{U_{j}^*\frac{\partial A}{\partial \beta}U_{i}}{\lambda_{i}-\lambda_j}.
\end{align*}
Combining the previous relations with \eqref{eq:Authetabeta}, we obtain
\begin{align*}
\frac{\partial A}{\partial \theta}\frac{\partial U_{i}}{\partial\beta}+\frac{\partial A}{\partial\beta}\frac{\partial U_{i}}{\partial\theta}
  &=\sum_{j\neq i}\frac{\partial A}{\partial \theta} U_{j}\frac{U_{j}^*\frac{\partial A}{\partial \beta}U_{i}}{\lambda_{i}-\lambda_j}+\sum_{j \neq i}\frac{\partial A}{\partial\beta}U_{j}\frac{U_{j}^*\frac{\partial A}{\partial \theta}U_{i}}{\lambda_{i}-\lambda_j}.
\end{align*}
Multiplying by $U^*_i$ in the previous identity, we get 
\begin{align}\label{eq:Uipartialmixed}
U_i^*\frac{\partial A}{\partial \theta}\frac{\partial U_{i}}{\partial\beta}+U_i^*\frac{\partial A}{\partial\beta}\frac{\partial U_{i}}{\partial\theta}
	&=\sum_{j\neq i}\frac{2\big(U_i^*\frac{\partial A}{\partial \theta} U_{j}\big)\big(U_{j}^*\frac{\partial A}{\partial \beta}U_{i}\big)}{\lambda_{i}-\lambda_j}.
\end{align}
Therefore, identity \eqref{eq:Hadamard2} follows from \eqref{eq:partiallambdathetabeta} and \eqref{eq:Uipartialmixed}. The proof is now complete.
\end{proof}

\begin{Lemma}\label{lem:K}
Consider the Kernel
\begin{align*}
K_{\rho}(x,y)
  &:=\sum_{q=0}^{\infty}U_{q}(x)U_{q}(y)\rho^{q}. 
\end{align*}
Then, for every $x,y\in(-2,2)$ and $\rho\in[0,1)$, the series defining $K_\rho(x,y)$ is absolutely convergent and
\begin{align}\label{eq:Kform}
K_{\rho}\big(x,y\big)
  &=\frac{1-\rho^2}{\rho^2(x-y)^2-xy\rho(1-\rho)^2+(1-\rho^2)^2}.
\end{align}
Furthermore, $K_{\rho}(x,y)\geq 0$ for all $x,y\in(-2,2)$ and $K_{\rho}$ is integrable over $(-2,2)^2$.
\end{Lemma}
\begin{proof}
For $x\in(-1,1)$, define $\widetilde{U}_q(x):=U_q\big(2x\big)$. It is not hard to verify that $(\widetilde{U}_q;q\in\mathbb{N})$ are the Chebyshev polynomials of the second kind on $[-1,1]$. Using the well-known formula
\begin{align*}
\widetilde{U}_q(x)
  &=\frac{(x+\textbf{i}\sqrt{1-x^2})^{q+1}-(x-\textbf{i}\sqrt{1-x^2})^{q+1}}{2\textbf{i}\sqrt{1-x^2}},
\end{align*}
and defining $a:=x+\textbf{i}\sqrt{1-x^2}$ and $b:=y+\textbf{i}\sqrt{1-y^2}$, we get 
\begin{align*}
\widetilde{U}_q(x)\widetilde{U}_q(y)\rho^{q}
  &=\frac{-\rho^{q}}{4\sqrt{(1-x^2)(1-y^2)}}(a^{q+1}-\overline{a}^{q+1})(b^{q+1}-\overline{b}^{q+1})\\
	&=\frac{-1}{4\rho\sqrt{(1-x^2)(1-y^2)}}((ab\rho)^{q+1}+(\overline{ab}\rho)^{q+1}-(a\overline{b}\rho)^{q+1}-(\overline{a}b\rho)^{q+1}).
\end{align*}
Observe that $|a|=|b|=1$, and thus, since $\rho\in(0,1)$, we can compute the sum over $q$ by means of the geometric series, i.e. 
\begin{align*}
\sum_{q=0}^{\infty}\widetilde{U}_q(x)\widetilde{U}_q(y)\rho^{q}
  &=\frac{-1}{4\rho\sqrt{(1-x^2)(1-y^2)}}\bigg(\frac{ab\rho}{1-ab\rho}+\frac{\overline{ab}\rho}{1-\overline{ab}\rho}-\frac{a\overline{b}\rho}{1-a\overline{b}\rho}-\frac{\overline{a}b\rho}{1-\overline{a}b\rho}\bigg)\\
	&=\frac{-1}{4\sqrt{(1-x^2)(1-y^2)}}\bigg(\frac{ab+\overline{ab}-2\rho}{|1-ab\rho|^2}-\frac{a\overline{b}+\overline{a}b-2\rho}{|1-a\overline{b}\rho|^{2}}\bigg)\\
	&=\frac{-1}{2\sqrt{(1-x^2)(1-y^2)}}\bigg(\frac{\mathfrak{R}(ab)-\rho}{|ab-\rho|^2}-\frac{\mathfrak{R}(a\overline{b})-\rho}{|a\overline{b}-\rho|^{2}}\bigg),
\end{align*}
where $\mathfrak{R}(z)$ means the real part of $z$. Indeed, we have that
\begin{equation}
\label{eq:KernelBound}
    \sum_{q=0}^{\infty} \left|\widetilde{U}_q(x)\widetilde{U}_q(y)\rho^{q}\right| \leq \frac{2}{\sqrt{(1-x^2)(1-y^2)}(1-\rho)^2},
\end{equation}
i.e., the series defining $K_\rho(x,y)$ is absolutely convergent. We can also easily verify that 
\begin{align*}
\mathfrak{R}(ab)
  &=xy-\sqrt{(1-x^2)(1-y^2)}\\
\mathfrak{R}(a\overline{b})
  &=xy+\sqrt{(1-x^2)(1-y^2)},
\end{align*}
and 
\begin{align*}
|ab-\rho|^2
  &=1-2xy\rho+2\rho\sqrt{(1-x^2)(1-y^2)}+\rho^2\\
|a\overline{b}-\rho|^2
  &=1-2xy\rho-2\rho\sqrt{(1-x^2)(1-y^2)}+\rho^2.
\end{align*}
Putting these identities together, we deduce
\begin{equation*}
    \sum_{q=0}^{\infty}\widetilde{U}_q(x)\widetilde{U}_q(y)\rho^{q} =\frac{1-\rho^2}{4\rho^2(x^2+y^2)-4xy\rho(1+\rho^2)+1-2\rho^2+\rho^4}.
\end{equation*}
From the previous analysis, it easily follows 
\begin{equation}
\label{eq:KernelReduction}
    K_{\rho}\big(x,y\big) =\frac{1-\rho^2}{\rho^2(x-y)^2-xy\rho(1-\rho)^2+(1-\rho^2)^2},
\end{equation}
from where we deduce the identity \eqref{eq:Kform}.

In order to establish positivity, assume that $K_\rho(x_0,y_0) < 0$ for some $x_0,y_0\in(-2,2)$. Observe that the denominator of the right hand side of \eqref{eq:KernelReduction} is a continuous function w.r.t.~$(x,y)\in(-1,1)^2$. Since $K_\rho(0,0)>0$, the denominator should vanish at some point $(x,y)\in(-1,1)^2$. However, this contradicts the inequality in \eqref{eq:KernelBound} as $K_\rho(x,y)$ should blow up. Therefore, $K_{\rho}(x,y)>0$ for all $x,y\in(-2,2)$. Furthermore, by the bound in \eqref{eq:KernelBound} we conclude that $K_\rho$ is integrable over $(-2,2)$. The proof is now complete.
\end{proof}

\noindent \textbf{Acknowledgements.} JCP  acknowledges support from  the Royal Society and CONACyT (CB-250590). This work was concluded whilst JCP was on sabbatical leave holding a David Parkin Visiting Professorship  at the University of Bath, he gratefully acknowledges the kind hospitality of the Department and University. This work was started when AJ was postdoctoral researcher jointly at the University of Luxembourg and the National University of Singapore.

\bibliographystyle{plain}
\bibliography{bibliography}

\end{document}